\documentclass[a4,11pt]{amsart}
\usepackage[margin=3.5cm]{geometry}
\usepackage[utf8]{inputenc}
\usepackage{tikz}
\usepackage{xcolor}
\usepackage{verbatim}
\usepackage{amsmath,amsthm,amssymb}
\usepackage{url}
\usepackage{graphicx} 
\usetikzlibrary{arrows.meta,
                chains,
                decorations.pathreplacing,calligraphy,
                positioning,
                quotes,
                shapes.geometric
                }

\newtheorem{theorem}{Theorem}[section]
\newtheorem{corollary}[theorem]{Corollary}
\newtheorem{proposition}[theorem]{Proposition}
\newtheorem{lemma}[theorem]{Lemma}
\newtheorem{construction}[theorem]{Construction}

\usepackage{tikz-cd} 

\usepackage{tikz-cd}
\usepackage{soul}
\usepackage{tkz-berge}

\definecolor{green}{HTML}{37cc57}
\definecolor{red}{HTML}{e81e5b}
\definecolor{blue}{HTML}{3d8bbf}

\DeclareMathOperator{\Aut}{Aut}
\DeclareMathOperator{\AGL}{AGL}
\DeclareMathOperator{\PSL}{PSL}
\DeclareMathOperator{\GL}{GL}

\DeclareMathOperator{\Sz}{Sz}

\DeclareMathOperator{\Cor}{Cor}
\DeclareMathOperator{\AG}{AG}
\DeclareMathOperator{\PGammaL}{P\Gamma L}
\DeclareMathOperator{\PG}{PG}
\DeclareMathOperator{\SL}{SL}
\DeclareMathOperator{\UH}{U_H}
\DeclareMathOperator{\PGammaU}{P \Gamma U}
\DeclareMathOperator{\AGammaL}{A \Gamma L}
\DeclareMathOperator{\PSU}{PSU}
\DeclareMathOperator{\Kv}{K}
\DeclareMathOperator{\Kn}{K_n}
\DeclareMathOperator{\K3}{K_3}
\DeclareMathOperator{\Alt}{A}
\DeclareMathOperator{\St}{S_3}
\DeclareMathOperator{\Ss}{S_6}
\DeclareMathOperator{\Sv}{S}
\DeclareMathOperator{\Sn}{S_n}
\DeclareMathOperator{\Mv}{M}
\DeclareMathOperator{\Mo}{M_{11}}

\DeclareMathOperator{\Mvd}{M_{22}}
\DeclareMathOperator{\ea}{ea}
\DeclareMathOperator{\Ct}{C_3}

\title{Geometries with trialities arising from linear spaces}
\author{Remi Delaby}
\address{Remi Delaby, Universit\'e Libre de Bruxelles, D\'epartement de Math\'ematique, C.P.216 - Alg\`ebre et Combinatoire, Boulevard du Triomphe, 1050 Brussels, Belgium.}
\email{remi.delaby@ulb.be}

\author{Dimitri Leemans}\thanks{This research was made possible thanks to an Action de Recherche Concert\'ee grant from the Communaut\'e Fran\c caise Wallonie-Bruxelles.}
\address{Dimitri Leemans, Universit\'e Libre de Bruxelles, D\'epartement de Math\'ematique, C.P.216 - Alg\`ebre et Combinatoire, Boulevard du Triomphe, 1050 Brussels, Belgium, Orcid number 0000-0002-4439-502X.}
\email{leemans.dimitri@ulb.be} 

\author{Philippe Tranchida}
\address{Philippe Tranchida, Max Planck Institute for Mathematics in the Sciences, Inselstrasse 22 D-04107 Leipzig, Orcid number 0000-0003-0744-4934.
}
\email{tranchida.philippe@gmail.com}

\date{\today}

\begin{document}

\maketitle

\begin{center}

\tikzset{node distance=2cm, every node/.style={circle,thin, fill=black}}

\begin{tikzpicture}[scale = 1]

\renewcommand*{\EdgeLineWidth}{1pt}

\node[fill, color=green, circle] (1) {};
\node[right of=1, fill, color=black, circle, scale = .2] (2) {};
\node[right of=2, fill, color=blue, circle] (3) {};
\node[below of=1, fill, color=black, circle, scale = .2] (4) {};
\node[below of=2, fill, color=black, circle, scale = .2] (5) {};
\node[below of=3, fill, color=black, circle, scale = .2] (6) {};
\node[below of=4, fill, color=black, circle, scale = .2] (7) {};
\node[below of=5, fill, color=red, circle] (8) {};
\node[below of=6, fill, color=black, circle, scale = .2] (9) {};
\draw[-]
(1) edge[blue,line width = 2] (2)
(2) edge[blue,line width = 2] (3)
(4) edge[black] (5)
(5) edge[black] (6)
(7) edge[black] (8)
(8) edge[black] (9)
(3) edge[black] (5)
(5) edge[black] (7)
(1) edge[black] (5)
(5) edge[black] (9)
(3) edge[black] (6)
(6) edge[black] (9)
(1) edge[black] (4)
(4) edge[black] (7)
(2) edge[black] (5)
(5) edge[black] (8)
(2) edge[black] (4)
(8) edge[green,line width = 2] (6)
(6) edge[black] (2)
(4) edge[red, line width = 2] (8)
(8) edge[red, line width = 2, out=315, in=315, looseness=2] (3)
(4) edge[black, out=225, in=225, looseness=2] (9)
(6) edge[green,line width = 2, out=45, in=45, looseness=2] (1)
(2) edge[black, out=135, in=135, looseness=2] (7);

\end{tikzpicture}

\end{center}

\begin{abstract}

A triality is a sort of super-symmetry that exchanges the types of the elements of an incidence geometry in cycles of length three. Although geometries with trialities exhibit fascinating behaviors, their construction is challenging, making them rare in the literature. To understand trialities more deeply, it is crucial to have a wide variety of examples at hand. In this article, we introduce a general method for constructing various rank-three incidence systems with trialities. Specifically, for any rank two incidence system $\Gamma$, we define its triangle complex $\Delta(\Gamma)$, a rank three incidence system whose elements consist of three copies of the flags (pairs of incident elements) of $\Gamma$. This triangle complex always admits a triality that cyclically permutes the three copies. We then explore in detail the properties of the triangle complex when $\Gamma$ is a linear space, including flag-transitivity, the existence of dualities, and connectivity properties. As a consequence of our work, this construction yields the first infinite family of thick, flag-transitive and residually connected geometries with trialities but no dualities.

\end{abstract}

\section{Introduction}

Incidence geometries are geometric objects formed of elements of different types together with an incidence relation between them. In the more classical examples, incidence geometries are formed of points, lines, faces, etc. Some incidence geometries allow for some type of super-symmetry that exchanges the role of their types. These types of symmetries are formally called correlations. The most classical example is the one of a duality or a polarity, which exchanges the role of points and hyperplanes in projective geometries or the role of points and faces in polyhedra. Higher order collineations have not been systematically studied. That being said, order three correlations, also called trialities did make some very noticeable appearances in the history of mathematics. 

The notion of triality appeared in papers of Study (see~\cite[Page 435]{Porteous}, see also~\cite{Study1} and~\cite{Study2}\footnote{See \url{http://neo-classical-physics.info/uploads/3/4/3/6/34363841/study-analytical_kinematics.pdf} 
for an english translation of~\cite{Study2}.}), where he used a quadric in a seven dimensional projective plane to describe physical motions. Cartan later gave the official name of triality~\cite{Cartan} to this phenomenon, and defined it in the broader context of Lie groups. This quadric was later investigated by Tits in \cite{tits1959trialite}. In \cite{trialityAbsolutes}, it is shown how to obtain interesting geometries from the absolute points and the moving lines of this same quadric under the actions of the triality. Freudenthal further explored triality in the context of Lie algebras \cite{Freudenthal}. The notion of triality also appears in the study of maps on surfaces, where a triality is obtained as a composition of the Dual operator and the Petrie dual operator (see \cite{abrams2022new}, \cite{jones2010maps}, \cite{wilson1979operators}). Recently, Leemans and Stokes \cite{LeemansStokes2019} showed how to construct coset geometries having trialities by using a group $G$ that has outer automorphisms of order three. They used this technique in~\cite{leemans2022incidence} to construct the first infinite family of coset geometries admitting trialities but no dualities. In \cite{trialitySuzuki}, the first infinite family of flag-transitive geometries admitting trialities but no dualities is given. The construction uses coset geometries arising from Suzuki groups $\Sz(q)$, where $q=2^{2e+1}$ with $e$ a positive integer and $2e+1$ is divisible by $3$. 
We refer to the introduction of \cite{trialitySuzuki} for a more detailed account of the history of trialities.

In this article, we present a new construction that produces a rank three geometry $\Delta(\Gamma)$ with trialities from any given rank two geometry $\Gamma$ (see Construction~\ref{construct}). The elements of this new geometry $\Delta(\Gamma)$ are three copies of the flags of $\Gamma$ and the trialities exchange cyclically these three copies. Since the chambers of $\Delta(\Gamma)$ correspond to triangles of $\Gamma$, we call $\Delta(\Gamma)$ the triangle complex of $\Gamma$. We then study in details the case where $\Gamma$ is a linear space. We show that $\Delta(\Gamma)$ is flag-transitive if and only if $\Gamma$ is transitive on its set of triple of non-collinear points (Theorem \ref{mastertheorem}). This prompts us to give a classification of such linear spaces. We achieve this in Theorem \ref{transitive}, using the previous classification of locally two-transitive linear spaces obtained in \cite{BDL2003}. We obtain the first new theorem of this paper which may be of interest to researchers studying linear spaces.

\begin{theorem}\label{transitive}
Assume $\Gamma$ is a flag-transitive linear space of $v$ points with a group $G$ acting transitively on the set $T(\Gamma)$ of ordered triples of non-collinear points of $\Gamma$. Then one of the following occurs:
\begin{enumerate}
\item $\Gamma=\PG(n,q)$, $v=\frac{q^{n+1}-1}{q-1}$, $\PSL(n+1,q)\trianglelefteq G \le \PGammaL(n+1,q)$ with $n\ge 2$.
\item $\Gamma=\PG(3,2)$, $v=15$ with $G\cong \Alt_7$. 
\item $\Gamma=\AG(n,q)$, $v=p^d=q^n$, $G=p^d:G_0$ with $\SL(n,q)\trianglelefteq G_0$, $q\ge 3$ and $n\ge 3$ or $\GL(n,q)\trianglelefteq G_0$, $q\ge 3$ and $n = 2$.
\item  $\Gamma$ is a hermitian unital $\UH(q)$, $v=q^3+1$, $G = \PGammaU(3,q)$ and $q=2$ or $4$.
\item $\Gamma$ is a circle and $G$ is a $3$-transitive permutation group:
     \begin{enumerate}
	\item $G=\Alt_v,G=\Sv_v$, $v\ge 3$.
	\item $G\trianglerighteq {\PSL}(2,q)$, $v=q+1$ and $G$ normalizes a sharply $3$-transitive permutation group.
	\item $G = \Mv_v$, $v=11,12,22,23,24$ or $G=\Aut(\Mvd)$, $v=22$.
	\item $G= \Mo$, $v=12$.
	\item $G= \Alt_7$, $v=15$.
        \item $G=\ea(2^n):G_{0}$, $v=2^n$, $G_0\trianglerighteq \SL(n,2)$ with $n\ge 2$.
	\item $G=\ea(2^4):\Alt_7$, $v=16$.
      \end{enumerate}
\end{enumerate}
The converse is also true. For any pair $(\Gamma,G)$ satisfying the conditions of one of the cases given above, the action of $G$ is transitive on $T(\Gamma)$.
\end{theorem}

We then investigate which of the geometries $\Gamma$ appearing in the classification theorem above yields a firm, residually connected and flag-transitive geometry $\Delta(\Gamma)$.
We obtain the following classification theorem.
\begin{theorem}\label{classification}
Assume $\Gamma$ is a linear space of $v$ points. The geometry $\Delta(\Gamma)$ obtained from Construction~\ref{construct} is firm, residually connected and flag-transitive if and only if $\Gamma$ is one of the following:
\begin{enumerate}
\item A projective plane $\PG(2,q)$ with $q\geq 2$ and $v=q^2+q+1$;
\item An affine plane $\AG(2,q)$ with $q\geq 3$ and $v=q^2$;
\item A hermitian unital $\UH(q)$ with $q=2$ or $4$ and $v=q^3+1$.
\end{enumerate}
Moreover, in each of these cases, $\Delta(\Gamma)$ admits dualities if and only $\Gamma$ does.
\end{theorem}
In particular, applying our construction to finite affine planes results in an infinite family of flag-transitive, residually connected and thick geometries with trialities but no dualities. While examples of infinite families of geometries with trialities but no dualities were given \cite{LeemansStokes2019} and \cite{trialitySuzuki}, in both cases, the geometries are constructed as coset geometries. This means that the elements of these geometries are defined as coset of some systems $(G_i)_{i\in I}$ of subgroups of some group $G$, and do not have simple geometric interpretation. Moreover, the existence of a triality and the lack of dualities for these geometries also relies on group theoretical argument. Indeed, the triality is induced by an outer automorphism of $G$ of order three, and the absence of outer automorphisms of order two in $G$ implies the impossibility for the geometry to have dualities. In our case, the definition of $\Delta(\Gamma)$, the existence of the triality and the lack of dualities are all based on geometric constructions and arguments.

The geometry $\Delta(\AG(2,3))$ is the smallest geometry with trialities but no dualities arising from our construction. It is a geometry of rank three with 108 elements, 36 of each type. It has $432$ maximal flags and can be considered as an orientable proper hypermap on a surface of genus 55. This hypermap is the hypermap RPH55.89 on Marston Conder's list of orientable proper hypermaps\footnote{
See \url{https://www.math.auckland.ac.nz/~conder/OrientableProperHypermaps101.txt}
}.

The paper is organized as follows.
In Section~\ref{sec:prelim}, we give the basic notions needed to understand the paper.
In Section~\ref{sec:triangle}, we give the construction of the triangle complex of a rank two incidence system and we study the link between automorphisms and correlations of $\Gamma$ and those of $\Delta(\Gamma)$ in details.
In Section~\ref{sec:linear}, we prove that a linear space $\Gamma$ yields a flag-transitive triangle complex $\Delta(\Gamma)$ if and only if there exists $G\leq Aut(\Gamma)$ that acts transitively on the set of ordered triples of non-collinear points of $\Gamma$. This leads us to prove Theorem~\ref{transitive} which classifies the linear spaces $\Gamma$ that are such that their triangle complex $\Delta(\Gamma)$ is flag-transitive.
In Section~\ref{sec:rc}, we study the connectedness and residual connectedness of $\Delta(\Gamma)$, ultimately allowing us to prove Theorem~\ref{classification}.
In Section~\ref{sec:diagrams}, we give the Buekenhout diagrams of the triangle complexes obtained in Theorem~\ref{classification}. Finally, in Section~\ref{sec:conclusion}, we give some concluding remarks and suggest some future work.

\section{Preliminaries}\label{sec:prelim}

Geometric structures are composed of objects and a relation that specifies how these objects are related to one another. The notion of incidence system formalizes this concept by providing an abstract framework to study these configurations. We refer to~\cite{BC2013} for an introduction to this subject and more generally to the subject of diagram geometry.

Let $I$ be a finite non empty set.
    A triple $\Gamma = (X,\star,t)$ is called an \textit{incidence system} over $I$ if
    \begin{enumerate}
        \item $X$ is a non empty set whose elements are called the \textit{elements} of $\Gamma$;
        \item $\star$ is a symmetric and reflexive relation on $X$. It is called the \textit{incidence relation} of $\Gamma$;
        \item $t$ is a map from $X$ to $I$, called the \textit{type map} of $\Gamma$, such that distinct elements $x,y \in X$ with $x \star y$ satisfy $t(x) \neq t(y)$. Elements of $t^{-1}(i)$ are called elements of type $i$.
    \end{enumerate}

The \textit{rank} of $\Gamma$ is the cardinality of the type set $I$.
A \textit{flag} in an incidence system $\Gamma$ over $I$ is a set of pairwise incident elements. The type of a flag $F$ is $t(F)$, that is the set of types of the elements of $F.$ A \textit{chamber} is a flag of type $I$. An incidence system $\Gamma$ is an \textit{incidence geometry} if all its maximal flags are chambers.

Let $F$ be a flag of $\Gamma$. An element $x\in X$ is {\em incident} to $F$ if $x\star y$ for all $y\in F$. The \textit{residue} of $\Gamma$ with respect to $F$, denoted by $\Gamma_F$, is the incidence system formed by all the elements of $\Gamma$ incident to $F$ but not in $F$. The \textit{rank} of a residue is equal to rank$(\Gamma)$ - $|F|$. For an element $x\in X$, we denote by $\text{Res}_{\Gamma}(x)$ the set of elements of $\Gamma_{\{x\}}$.

A geometry $\Gamma$ is \textit{firm} (respectively, \textit{thick}) if every flag of type other than $I$ is contained in at least two (respectively, three) distinct chambers of $\Gamma$. It is called \textit{thin} if every flag of type $I \backslash \{i\}$ for some $i \in I$ is contained in exactly two chambers of $\Gamma$.

The \textit{incidence graph} of $\Gamma$ is a graph with vertex set $X$ and where two elements $x$ and $y$ are connected by an edge if and only if $x \star y$ and $t(x) \neq t(y)$. Whenever we talk about the distance between two elements $x$ and $y$ of a geometry $\Gamma$, we mean the distance in the incidence graph of $\Gamma$ and simply denote it by $d_\Gamma(x,y)$, or even $d(x,y)$ if the context allows.
The geometry $\Gamma$ is \textit{residually connected} when the incidence graphs of all of its residues of rank at least $2$ are connected.

Let $\Gamma = \Gamma(X,\star,t)$ be an incidence system over the type set $I$. A correlation of $\Gamma$ is a bijection $\phi$ of $X$ respecting the incidence relation $\star$ and such that, for every $x,y \in X$, if $t(x) = t(y)$ then $t(\phi(x)) = t(\phi(y))$. If, moreover, $\phi$ fixes the types of every element (i.e $t(\phi(x)) = t(x)$ for all $x \in X$), then $\phi$ is said to be an automorphism of $\Gamma$. The \emph{type} of a correlation $\phi$ is the permutation it induces on the type set $I$. A correlation of type $(i,j)$ is called a {\em duality} if it has order $2$. A correlation of type $(i,j,k)$ is called a {\em triality} if it has order $3$. The group of all correlations of $\Gamma$ is denoted by $\Cor(\Gamma)$ and the automorphism group of $\Gamma$ is denoted by $\Aut(\Gamma)$. Remark that $\Aut(\Gamma)$ is a normal subgroup of $\Cor(\Gamma)$ since it is the kernel of the action of $\Cor(\Gamma)$ on $I$.

A geometry $\Gamma$ such that the action of $\Aut(\Gamma)$ on the set of chambers if transitive is called \textit{flag-transitive}. Let $G \leq \Aut(\Gamma)$. The pair $(\Gamma,G)$ is called $(2T)_1$ if for every flag $F$ of rank $|I|-1$, the stabilizer of $F$ in $G$ acts two-transitively on the elements of the residue $\Gamma_F$ ~\cite{srzk2}.

Francis Buekenhout introduced a diagram associated to incidence geometries ~\cite{buekenhout2013diagram}. His idea was to associate to each rank two residue a set of three integers giving information on its incidence graph.
Let $\Gamma$ be a rank $2$ geometry. We can consider $\Gamma$ to have type set $I = \{P,L\}$, where $P$ and $L$ stand for points and lines. The {\em point-diameter}, denoted by $d_P(\Gamma) = d_P$, is the largest integer $k$ such that there exists a point $p \in P$ and an element $x \in \Gamma$ with $d(p,x) = k$. Similarly the {\em line-diameter}, denoted by $d_L(\Gamma) = d_L$, is the largest integer $k$ such that there exists a line $l \in L$ and an element $x \in \Gamma$ with $d(l,x) = k$. Finally, the \textit{gonality} of $\Gamma$, denoted by $g(\Gamma) = g$ is half the length of the smallest circuit in the incidence graph of $\Gamma$.

If a rank $2$ geometry $\Gamma$ has $d_P = d_L = g = n$ for some natural number $n$, we say that it is a \textit{generalized $n$-gon}. Generalized $2$-gons are also called generalized digons. They are in some sense trivial geometries since all points are incident to all lines. Their incidence graphs are complete bipartite graphs. Generalized $3$-gons are projective planes.

Consider a rank $2$ geometry $\Gamma$ over $I = \{P,L\}$ such that the following axioms are fulfilled.
\begin{enumerate}
    \item Every line is incident to at least two points.
    \item Every point is incident to at least two lines.
    \item Every pair of distinct points $p,q$ is incident to one and only one line.
\end{enumerate}
Such an incidence geometry $\Gamma = (\mathcal{P} \sqcup \mathcal{L},\star,t ) $ over $I = \{P,L\}$ is called a \textit{linear space} where $t$ is defined in the obvious way (see \cite[Chapter 1, Section 2.5]{Handbook}). We decided to add axiom (2) to the definition of \cite[Chapter 1, Section 2.5]{Handbook} in order to avoid the degenerate case where  $\Gamma$ would have a unique line.

Let $\Gamma$ be a geometry over $I$.  The \textit{Buekenhout diagram} (or diagram for short) $D$ for $\Gamma$ is a graph whose vertex set is $I$. Each edge $\{i,j\}$ is labeled with a collection $D_{ij}$ of rank $2$ geometries. We say that $\Gamma$ belongs to $D$ if every residue of rank $2$ of type $\{i,j\}$ of $\Gamma$ is one of those listed in $D_{ij}$ for every pair of $i \neq j \in I$. In most cases, we use conventions to turn a diagram $D$ into a labeled graph. The most common convention is to not draw an edge between two vertices $i$ and $j$ if all residues of type $\{i,j\}$ are generalized digons, and to label the edge $\{i,j\}$ by a natural integer $n$ if all residues of type $\{i,j\}$ are generalized $n$-gons. It is also common to omit the label when $n=3$.
If the edge $\{i,j\}$ is labeled by a triple $(d_{ij},g_{ij},d_{ji})$ it means that every residue of type $\{i,j\}$ had $d_P = d_{ij}, g = g_{ij}, d_L = d_{ji}$. We can also add information to the vertices of a diagram. 
We can label the vertex $i$ with the number $n_i$ of elements of type $i$ in $\Gamma$. Moreover, if for all flags $F$ of co-type $i$, we have that $|\Gamma_F| = s_i +1$, we will also label the vertex $i$ with the integer $s_i$.

\section{The triangle complex of a rank two geometry}\label{sec:triangle}

In this section, we start by defining the main construction of this article, that is, the triangle complex $\Delta(\Gamma)$ of a rank $2$ incidence system $\Gamma$. We then study in details the geometry $\Delta(\Gamma)$ when $\Gamma$ is a thick linear space. First, we pay attention to the automorphisms and correlations of $\Delta(\Gamma)$. It turns out that, except for the trialities, they are all induced by automorphisms or correlations of $\Gamma$. Finally, we investigate separately the case of linear spaces with two points on each lines, also known as circles. 

\begin{construction}\label{construct}
Let $\Gamma = (\mathcal{P} \sqcup \mathcal{L},\star,t)$ 
be a rank two geometry. The {\em triangle complex} $\Delta(\Gamma) = (X_{\Delta(\Gamma)},\star_{\Delta(\Gamma)},t_{\Delta(\Gamma)})$ over $I = \{1,2,3\}$ is the rank three incidence system constructed from $\Gamma$ in the following way: 

\begin{enumerate}
    \item The set $X_{\Delta(\Gamma)}$ of elements of $\Delta(\Gamma)$ is the set of all the triples $(p,L,i)$ with $p \in \mathcal P$, $L \in \mathcal L$, $i \in \{ 1,2,3 \}$ satisfying $p \star L$.
    \item The incidence relation $\star_{\Delta(\Gamma)}$ is defined by $(p,L,i) \star_{\Delta(\Gamma)} (p',L',i \bmod{3} +1)$ if and only if $\text{Res}_\Gamma(L') \cap \text{Res}_\Gamma(L) = \{ p \}$ and $p \neq p'$. 
    \item The type function $t_{\Delta(\Gamma)} \colon X_{\Delta(\Gamma)} \mapsto \{ 1, 2 , 3 \}$ is defined by $t_{\Delta(\Gamma)} ( (p,L,i) ) = i$.
\end{enumerate}    
\end{construction}

Simply put, the elements of $\Delta(\Gamma)$ are three disjoint copies of the flags of $\Gamma$. With the incidence relation defined above, it is easy to see that chambers of $\Delta(\Gamma)$ correspond to triangles in $\Gamma$. This motivates the name triangle complex. Before proceeding, we prove a few elementary results about $\Delta(\Gamma)$.

We first give a necessary condition for the incidence system $\Delta(\Gamma)$ to be a geometry. Note that from now on, we will call the elements of $\mathcal{P}$ the points of $\Gamma$ and the elements of $\mathcal{L}$ the lines of $\Gamma$.

\begin{proposition}\label{g3impliesgeom}
    Let $\Gamma = (\mathcal{P} \sqcup \mathcal{L},\star,t)$ be a rank two geometry. If $\Delta(\Gamma)$ is a geometry, then the gonality of $\Gamma$ is at most three.
\end{proposition}

\begin{proof}
    Let $p_1 \in \mathcal P$ be any point of $\Gamma$. Since $\Gamma$ is a geometry there must exist a line $L_1 \in \mathcal L$ such that $\{ p_1,L_1 \}$ is a flag in $\Gamma$. Thus, the triple $(p_1,L_1,1)$ is an element of $\Delta(\Gamma)$. If $\Delta(\Gamma)$ is a geometry, the flag $\{(p_1,L_1,1)\}$ can be completed in a chamber $\{ (p_1,L_1,1) , (p_2,L_2,2) , (p_3,L_3,3) \} $.
    Notice that this implies that $p_1,L_2,p_2,L_3,p_3,L_1,p_1$ is an incidence chain in $\Gamma$, so that the gonality of $\Gamma$ is at most three as claimed. 
\end{proof}

The previous proposition motivates us to look at linear spaces.
Indeed, linear spaces are rank two geometries with gonality three. They are thus a potential good source of inputs for Construction~\ref{construct}. 
\begin{proposition}
    If $\Gamma$ is a linear space, the incidence system $\Delta(\Gamma)$ is an incidence geometry.
\end{proposition}
\begin{proof}
    It suffices to observe that, by definition, a linear space has always at least two lines and all its lines have at least two points. Hence any flag of $\Delta(\Gamma)$ can be completed in a chamber. 
\end{proof}

\begin{proposition}\label{triality} 
    Let $\Gamma$ be a rank $2$ geometry. The triangle complex $\Delta(\Gamma)$ admits trialities.
\end{proposition}

\begin{proof}
    The application $\tau \colon \Delta(\Gamma) \rightarrow \Delta(\Gamma)$ defined by $\tau( (p,L,i) ) = (p,L,i \bmod{3} +1 )$ is a triality.
\end{proof}    

We call the triality $\tau$ above the {\em canonical triality} of $\Delta(\Gamma)$.

\subsection{A criterion for the flag-transitivity of $\Delta(\Gamma)$}\label{sec:FT}

We now start our investigation of automorphisms and correlations of $\Delta(\Gamma)$. The main technicality is to show that an automorphism $\phi \in \Aut(\Delta(\Gamma))$ acts on $\Delta(\Gamma)$ in a geometric way. For example, we will show that the images by $\phi$ of two elements $(p,L_1,i)$ and $(p,L_2,i)$ of $\Delta(\Gamma)$ that share a common point $p$ must still be two elements sharing a common point. While it is very tempting to assume such a behavior, it has to be carefully proved. This is achieved in Lemmas \ref{welldef} and \ref{bijection}.

Let $\Gamma = (\mathcal{P} \sqcup \mathcal{L},\star,t)$ be a linear space and let $\Delta(\Gamma)$ be the geometry constructed from $\Gamma$ using Construction~\ref{construct}. Given two points $p,q \in \mathcal P$ with $p \neq q$ we define the line $L := \big[ p,q \big]$ to be the only element of $\mathcal L$ incident to both $p$ and $q$.


\begin{lemma}\label{lem:m and n}
    Let $\Gamma = (\mathcal{P} \sqcup \mathcal{L},\star,t)$ be a finite linear space such that $\Delta(\Gamma)$ is flag transitive. Then every line of $\Gamma$ is incident to the same number of points and every point of $\Gamma$ is incident to the same number of lines.
\end{lemma}

\begin{proof}

Let $p , p' \in \mathcal{P}$ with $p\neq p'$. Construct $L = \big[p,p'\big]$. Counting the number of elements of type $3$ incident to $(p,L,1)$ and $(p',L,1)$ we get that

\begin{center}
    $\sum \limits_{q \in L/\{p,p'\}} (| \text{Res}_\Gamma(q)|-1) + | \text{Res}_\Gamma(p')|-1= \sum \limits_{q \in L/\{p,p'\}} (|\text{Res}_\Gamma(q)|-1) + | \text{Res}_\Gamma(p)|-1$.
\end{center}
It follows that $| \text{Res}_\Gamma(p)| = | \text{Res}_\Gamma(p')| = m$.

Let $L, L' \in \mathcal{L}$ with $L \neq L'$. Construct $p \in \text{Res}_\Gamma(L)/\text{Res}_\Gamma(L')$ and $p' \in \text{Res}_\Gamma(L')/\text{Res}_\Gamma(L)$.
Counting the number type $3$ elements incident to $(p,L,1)$ and $(p',L',1)$ we get that
\begin{center}
    $m(|\text{Res}_\Gamma(L)|-1) = m (|\text{Res}_\Gamma(L')|-1)$.
\end{center}
\end{proof}

Therefore, whenever $\Delta(\Gamma)$ is flag transitive, we will use the letters $m$ and $n$ to be $m = | \text{Res}_\Gamma(p)|$ for any $p \in \mathcal{P}$ and $n = | \text{Res}_\Gamma(L)|$ for any $L \in \mathcal{L}$.

The next lemma states that automorphisms of $\Delta(\Gamma)$ are well behaved in the sense that they send pencils of lines to pencils of lines and sets of collinear points to sets of collinear points.
For clarity, we introduce the following two projections maps 
$$\pi_1 \colon X_{\Delta(\Gamma)} \rightarrow \mathcal P : (p,L,i)  \mapsto p$$
$$\pi_2 \colon X_{\Delta(\Gamma)} \rightarrow \mathcal L :  (p,L,i)  \mapsto L$$
that will be used constantly in the proofs of this section.
\begin{lemma}\label{welldef}
    Let $\Gamma = (\mathcal{P} \sqcup \mathcal{L},\star,t)$ be a finite thick linear space such that $\Delta(\Gamma)$ is a flag-transitive geometry and let $\phi \in \Aut(\Delta(\Gamma))$. Then,
    \begin{enumerate}
\item   For every $p \in \mathcal P$ and for every $L,L' \in \text{Res}_\Gamma(p)$, we have that $\pi_1( \phi( (p,L,i) ) ) = \pi_1(\phi( (p,L',i) ))$. 
\item   For every $L \in \mathcal L$ and for every $p,p' \in \text{Res}_\Gamma(L)$, we have that $\pi_2( \phi( (p,L,i) ) ) = \pi_2(\phi( (p',L,i) ))$.
    \end{enumerate}
\end{lemma}

\begin{proof}
    Since $\Delta(\Gamma)$ is flag-transitive, we know that each line of $\Gamma$ is incident to the same number of points of $\Gamma$ and vice-versa.
    Let thus $n $ be the number of points per line in $\Gamma$ and let $m$ be the number of lines per point in $\Gamma$. Moreover, as $\Gamma$ is thick, we also know that both $m$ and $n$ are greater than or equal to $3$.
    
    
    By Proposition~\ref{triality}, without loss of generality, we will assume that $i= 1$. Assume by contradiction that (1) does not hold. 
    Then, there must exist two elements $F_1 = (p,L_1,1)$ and $F_2 = (p,L_2,1)$ in $\Delta(\Gamma)$ such that their images by $\phi$ do not share a common point. In other words, we have that $\pi_1( \phi(F_1) ) \neq \pi_1( \phi(F_2) )$. There should be as many type $2$ elements incident to both $F_1$ and $F_2$ as there are type $2$ elements incident to both $\phi(F_1)$ and $\phi(F_2)$. We will use a counting argument to show that this is in fact impossible when $\pi_1( \phi(F_1) ) \neq \pi_1( \phi(F_2) )$. The number of elements of type $2$ that are incident to both $F_1$ and $F_2$ is $(m - 2)(n - 1)$. 
    The same should then be true for the images $\phi(F_1)$ and $\phi(F_2)$. But we claim that the number of type $2$ elements incident to both $\phi(F_1)$ and $\phi(F_2)$ is either $n - 2$ or $0$. Indeed, since $\phi(F_1)$ and $\phi(F_2)$ do not share a common vertex, an element $F = (q,L,2) \in \Delta(\Gamma)$ is incident to $\phi(F_1)$ and $\phi(F_2)$ if and only if $L = \big[ \pi_1(\phi(F_1)),\pi_1(\phi(F_2)) \big] \neq \pi_2(\phi(F_1)) \neq \pi_2(\phi(F_2))$ and $q \notin \{ \pi_1(\phi(F_1)), \pi_1(\phi(F_2)) \}$.
    Hence, there are either $(n-2)$ or $0$ elements of type $2$ incident to both $\phi(F_1)$ and $\phi(F_2)$. Since there were $(m - 2)(n - 1)$ elements of type $2$ incident to both $F_1$ and $F_2$, this would imply that $(m-2)(n-1)$ is equal to either $n-2$ or $0$. As both $m$ and $n$ are at least $3$, this leads to a contradiction with our hypothesis. Hence, we conclude that $\pi_1( \phi(F_1) ) = \pi_1( \phi(F_2) )$.

    Assume now by contradiction that (2) is not true. We proceed in a similar fashion as for (1).
    Let $F_1 = (p_1,L,1)$ and $F_2 = (p_2,L,1)$ be two elements of $\Delta(\Gamma)$ such that $\pi_2(\phi(F_1)) \neq \pi_2(\phi(F_2))$. The number of elements of type $3$ incident to both $F_1$ and $F_2$ is $(n-2)(m-1)$. 
    If there is no point in $\Gamma$ incident to both $\pi_2(\phi(F_1))$ and $\pi_2(\phi(F_2))$, there are no elements of type $3$ incident to both $\pi_2(\phi(F_1))$ and $\pi_2(\phi(F_2))$. Suppose instead that $q$ is the only point of $\Gamma$ incident to both $\pi_2(\phi(F_1))$ and $\pi_2(\phi(F_2))$. If $q = \pi_1(\phi(F_1))$ or $\pi_1(\phi(F_2))$, we again have no elements of type $3$ incident to both $\pi_2(\phi(F_1))$ and $\pi_2(\phi(F_2))$. If not, the elements of type $3$ incident to both $\pi_2(\phi(F_1))$ and $\pi_2(\phi(F_2))$ are precisely the $(q,L,3)$ with $L \neq L_1,L_2$. There are $m-2$ such elements. This would then mean that $(n-2)(m-1)$ is equal to either $m-2$ or $0$, which again leads to a contradiction. We can thus conclude that $\pi_2(\phi(F_1)) = \pi_2(\phi(F_2))$.
\end{proof}

    
    Lemma \ref{welldef} suggests that given any $\phi \in \Aut(\Delta(\Gamma))$, we can recover a well-defined 
    automorphism of $\Gamma$ which fully characterizes $\phi$. For $p \in \mathcal P$, let us define $f_{\phi,i}(p) := \pi_1( \phi((p,L,i)))$ where $L$ is any line such that $p \in L$. Similarly, define  $g_{\phi,i}(L) := \pi_2( \phi((p,L,i)))$ where $p \in L$ is any point incident to $L$. By Lemma~\ref{welldef}, the functions $f_{\phi,i}$ and $g_{\phi,i}$ are well defined.

We now want to prove that the action of $\phi \in \Aut(\Delta(\Gamma))$ is the same in every copy of $\Gamma$ and comes from an automorphism of $\Gamma$. This is achieved by the next three results.

\begin{lemma}\label{bijection}
    Let $\Gamma = (\mathcal{P} \sqcup \mathcal{L},\star,t)$ be a finite thick linear space such that $\Delta(\Gamma)$ is flag-transitive. The functions $f_{\phi,i} \colon \mathcal P \rightarrow \mathcal P$ and $g_{\phi,i} \colon \mathcal L \rightarrow \mathcal L$ induced by $\phi \in \Aut(\Delta(\Gamma))$ as above are bijections.
\end{lemma}
\begin{proof}
Let $p \in \mathcal P$ be any point of $\Gamma$. Choose $L \in \mathcal L$ such that $p \star L$. Then $(p,L,i) \in X_{\Delta(\Gamma)}$. Since $\phi$ is an automorphism, there exists $(p',L',i)$ such that $\phi((p',L',i)) = (p,L,i)$ thus $f_{\phi,i} (p') = p$. since $\mathcal P$ is finite we conclude that $f_{\phi,i}$ is a bijection. The same argument shows that $g_{\phi,i}$ is also a bijection.
\end{proof}

\begin{proposition}\label{samefunctions}
Let $\Gamma = (\mathcal{P} \sqcup \mathcal{L},\star,t)$ be a finite thick linear space such that $\Delta(\Gamma)$ is flag-transitive. Then we have $f_{\phi,1} = f_{\phi,2} = f_{\phi,3}$ and $g_{\phi,1} = g_{\phi,2} = g_{\phi,3}$.
\end{proposition}

\begin{proof}

Let $p \in \mathcal P$ and $L \in \mathcal L$. As $\Delta(\Gamma)$ is flag-transitive, we can assume without loss of generality that $\phi( (p,L,1) ) = (p,L,1)$. Then, for any $p' \star L$, $p'\neq p$, there exists $L' \in \mathcal L$ such that $(p,L,1) \star_{\Delta(\Gamma)} (p',L',3)$. Since $\phi$ is an automorphism, we have that $(p,L,1) \star_{\Delta(\Gamma)} \phi( (p',L',3) )$. 
This means $f_{\phi,3} (p') \star L$ with $f_{\phi,3} (p') \neq p$. Since this is true for every $p'$ incident to $L$ with $p'\neq p$, and since, by Lemma~\ref{bijection}, $f_{\phi,3}$ is a bijection, the restriction of $f_{\phi,3}$ to $\{p'\in \mathcal P\mid p'\neq p \text{ and } p'\star L\}$ is a bijection. 
Now, keep the line $L$ and pick $p'' \star L$ with $p'' \neq p$ and follow the same process. We find that $f_{\phi,3} (p) \star L$, thus $f_{\phi,3} (p) = p = f_{\phi,1} (p)$.
The same reasoning allows us to prove that $f_{\phi,1} = f_{\phi,2} = f_{\phi,3}$ and $g_{\phi,1} = g_{\phi,2} = g_{\phi,3}$, where for the latter three function, we need to dualize the augments of this proof by considering pencils of points. 
\end{proof}

\begin{proposition}\label{masterproposition}
    Let $\Gamma = (\mathcal{P} \sqcup \mathcal{L},\star,t)$ be a finite thick linear space such that $\Delta(\Gamma)$ is flag-transitive. The function $F_\phi \colon X_ \Gamma \mapsto X_\Gamma$ defined by
    $$F_{\phi}(x) = 
    \begin{cases}
      f_{\phi,1}(x) & \text{if } x \in \mathcal P \\
      g_{\phi,1}(x) & \text{if } x \in \mathcal L
    \end{cases}
$$
is an automorphism of $\Gamma$.
Moreover, for all $p \in \mathcal P$, $L \in \mathcal L$ and $1 \leq i \leq 3$, we have that $\phi(p,L,i) = (F_\phi(p) , F_\phi(L) , i)$.
\end{proposition}

\begin{proof}
    We already know $F_\phi$ is a bijection by Lemma~\ref{bijection}. Let $p \in \mathcal P$ and $L \in \mathcal L$ such that $p \star L$. Choose $L' \neq L$ such that $p \star L'$ and $p' \star L$ with $p\neq p'$. We then have that $(p,L',1) \star_{\Delta(\Gamma)} (p',L,2)$. 
    Applying $\phi$, we get $f_{\phi,1}(p) \star g_{\phi,2}(L)$ but by Proposition~\ref{samefunctions}, we can conclude that $f_{\phi,1}(p) \star g_{\phi,1}(L)$ and hence $F_{\phi}(p) \star F_\phi(L)$.
    
    For the second part, assume that $p$ and $L$ are not incident in $\Gamma$. Pick $p' \star L$ and construct the line $L' = \big[p,p'\big]$. Then $(p',L,1) \star_{\Delta(\Gamma)} (p,L',2)$, thus $f_{\phi,2}(p)$ is not incident with $g_{\phi,1}(L)$. Hence, $f_{\phi,1}(p)$ is not incident with $g_{\phi,1}(L)$ which allows us to conclude that $F_{\phi} \in \text{Aut}(\Gamma)$.
\end{proof}

Armed with the results obtained so far, we are now ready to prove one of the main results of this section.

\begin{theorem}\label{mastertheorem}
    Let $\Gamma = (\mathcal{P} \sqcup \mathcal{L},\star,t)$ be a finite thick linear space such that $\Delta(\Gamma)$ is a geometry. Then $\Delta(\Gamma)$ is a flag transitive geometry if and only if $\text{Aut}(\Gamma)$ is transitive on $T(\Gamma)$, the set of ordered triples of non-collinear points of $\Gamma$.
\end{theorem}

\begin{proof}
    Assume $\text{Aut}(\Gamma)$ acts transitively on $T(\Gamma)$. 
    Let $C_1 = \{ (p_1,L_1,1) , (p_2,L_2,2) , (p_3,L_3,3) \}$ and $C_2 = \{ (p_1',L_1',1) , (p_2',L_2',2) , (p_3',L_3',3) \}$ be two chambers of $\Delta(\Gamma)$. The associated triples of points of both chambers are in $T(\Gamma)$. Take $g \in \Aut(\Gamma)$ satisfying $g(p_j) = p_j'$ for $j=1, 2, 3$. The induced automorphism $\phi_g \in \text{Aut}(\Delta(\Gamma))$ defined by $\phi_g (p,L,i) = (g(p),g(L),i)$ sends $C_1$ to $C_2$. Hence $\Delta(\Gamma)$ is flag-transitive.

    Assume $\Delta(\Gamma)$ is flag transitive. Let $T_1 = (p_1,p_2,p_3)$ and $T_2 = (p_1',p_2',p_3')$ in $T(\Gamma)$. Construct respectively two chambers $C_1$ and $C_2$ of $\Delta(\Gamma)$, the following way:
    \begin{center}
    $C_1 = \{ (p_1,\big[ p_1, p_3\big],1) , (p_2,\big[ p_2, p_1\big],2) , (p_3,\big[ p_2, p_3\big],3) \}$,\\
    $C_2 = \{ (p_1',\big[ p_1', p_3'\big],1) , (p_2',\big[ p_2', p_1'\big],2) , (p_3,\big[ p_2', p_3'\big],3) \}$.
    \end{center}

    Let $\phi \in \text{Aut}(\Delta(\Gamma))$ such that $\phi(C_1) = C_2$. By Proposition~\ref{masterproposition}, there exists $F_{\phi} \in \text{Aut}(\Gamma)$ such that $\phi(p,L,i) = (F_\phi(p) , F_\phi(L) , i)$. Thus, $F_\phi(p_j) = p_j'$ for $j=1, 2, 3$ and $\Aut(\Gamma)$ is transitive on $T(\Gamma)$.
\end{proof}



So far, we have focused only on automorphisms of both $\Gamma$ and $\Delta(\Gamma)$. We now take a look at dualities in both these geometries. 

\begin{proposition}\label{duality}
Suppose that $\Gamma = (\mathcal{P} \sqcup \mathcal{L},\star,t)$ is a finite thick linear space and that $\Delta(\Gamma)$ is flag-transitive. Then, $\Gamma$ admits dualities if and only if $\Delta(\Gamma)$ admits dualities.
\end{proposition}

\begin{proof}
    Suppose first that there exists a duality $\alpha$ of $\Gamma$. This means that $\alpha$ sends points to lines and lines to points while preserving incidence. We define a correlation $\beta$ of $\Delta(\Gamma)$ as follows:

 \begin{equation*}
        \beta(p,L,i) =
        \begin{cases}
            (\alpha(L),\alpha(p), 1), \textnormal{ if } i = 1 \\
            (\alpha(L),\alpha(p), 3), \textnormal{ if } i= 2  \\
            (\alpha(L),\alpha(p), 2), \textnormal{ if } i= 3  \\
        \end{cases}
    \end{equation*}
    It is rather straightforward to check that $
    \beta$ is indeed a correlation of order two, and hence a duality. For example, we have that $(p,L,1) \star_{\Delta(\Gamma)} (p',L',2)$ if and only if $p$ is different from $p'$ and $p$ is is on both $L$ and $L'$. This holds if and only if $\alpha(L)$ is different from $\alpha(L')$ and $\alpha(L')$ is on both $\alpha(p)$ and $\alpha(p')$. Hence, we get that $(p,L,1) \star_{\Delta(\Gamma)} (p',L',2)$ if and only if $(\alpha(L'),\alpha(p'),3) \star_{\Delta(\Gamma)} (\alpha(L), \alpha(p),1)$.

    Suppose now that there exists a duality $\beta$ of $\Delta(\Gamma)$. Without loss of generality, suppose that $\beta$ exchanges elements of type $2$ and $3$ of $\Delta(\Gamma)$. We will also denote by $\beta$ the map on $\{1,2,3\}$ induced by the duality, meaning that $\beta$ fixes $1$ and exchanges $2$ and $3$. The statements and proofs of Lemmas~\ref{welldef} and~\ref{bijection} and Proposition~\ref{masterproposition} can be adapted to the setting of dualities. More precisely, it can be shown that if there exists a duality $\beta$ of $\Delta(\Gamma)$ as above, then the number $m$ of lines per points in $\Gamma$ must be equal to the number $n$ of points per line and for $p \in \mathcal{P}$, and for any $L,L'\in Res_\Gamma(p)$, we have that $\pi_2(\beta((p,L,\beta(i)))) = \pi_2(\beta((p,L',\beta(i))))$ and for $L \in \mathcal{L}$ and any $p,p' \in Res_\Gamma(L)$, we have that $\pi_1(\beta((p,L,\beta(i)))) = \pi_1(\beta((p',L,\beta(i))))$. Intuitively, this means that $\beta$ sends pencils of lines to collinear points and vice versa. The proof of this statement is almost identical to the proof of Lemma \ref{welldef}. The only difference is that since $\beta$ exchanges the roles of elements of type $2$ and $3$, we need to compare the number of elements of type $2$ incident to pairs $F_1,F_2$ of elements of $\Gamma$ with the number of elements of type $3$ (not of type $2$) of their images $\beta(F_1)$ and $\beta(F_2)$. We can then proceed to show that there are bijections $f_{\beta,i} \colon \mathcal{P} \to \mathcal{L}$ and $g_{\beta,i} \colon \mathcal{L} \to \mathcal{P}$ and that these functions actually do not depend on $i$. Finally, we can show that $F_{\beta}(x) = $
    $
    \begin{cases}
      f_{\beta,1}(x) & \text{if } x \in \mathcal P \\
      g_{\beta,1}(x) & \text{if } x \in \mathcal L
    \end{cases}
    $
    is a duality of $\Gamma$.
\end{proof}

\begin{proposition}\label{prop:dual}
    Let $\Gamma$ be a finite thick flag transitive linear space such that $\Delta(\Gamma)$ is flag transitive, and let $ \langle \tau \rangle \cong \Ct$ be the subgroup of $\text{Cor}(\Delta(\Gamma))$ generated by the canonical triality $\tau$. Then $\langle \tau \rangle$ is always normal in $\text{Cor}(\Delta(\Gamma))$. Moreover, we have that $\text{Cor}(\Delta(\Gamma)) \cong \Ct \times \text{Cor}(\Gamma)$ if $\Gamma$ does not admit dualities, and $\text{Cor}(\Delta(\Gamma)) \cong \Ct \rtimes \text{Cor}(\Gamma)$ if $\Gamma$ admits a duality.

\end{proposition}

\begin{proof}
    Let $\tau$ as in Proposition~\ref{triality} and $\alpha$ a duality of $\Gamma$, then $\langle \tau \rangle \trianglelefteq \Cor(\Delta(\Gamma))$ since $\langle (1,2,3) \rangle \trianglelefteq \St$ and elements of $\Cor(\Delta(\Gamma))$ have the same action in each of the three copies of $\Gamma$ (see Proposition~\ref{samefunctions} and Proposition~\ref{duality}). If $\Gamma$ admits dualities then the following short exact sequence

    \begin{center}
        \begin{tikzcd}%
        1 \arrow[r] & \Aut(\Delta(\Gamma)) \arrow[r] & \Cor(\Delta(\Gamma)) \arrow[r]  & \St \arrow[r] & 1
    \end{tikzcd}%
    \end{center}
    is a split extension with $s \colon \St \rightarrow \Cor(\Delta(\Gamma))$ defined by $s(\mu)(p,L,i) = (p,L,\mu(i))$ if $\sigma(\mu) = 1$ and $s(\mu)(p,L,i) = (\alpha(L),\alpha(p),\mu(i))$ if $\sigma(\mu) = -1$ where $\sigma \colon \St \rightarrow \{ -1, 1 \}$ is the sign function. Thus $\Cor(\Delta(\Gamma)) \cong \Aut(\Delta(\Gamma)) \rtimes \St$. Moreover, since $\langle \tau \rangle$ is normal, $\Cor(\Delta(\Gamma)) = \langle \tau \rangle \Cor(\Gamma)$ and $\langle \tau \rangle \cap \Cor(\Gamma) = \{1\}$, we also have that $\Cor(\Delta(\Gamma)) \cong \Ct \rtimes \Cor(\Gamma)$. Meanwhile if $\Gamma$ does not admit dualities then the short exact sequence

    \begin{center}
        \begin{tikzcd}%
        1 \arrow[r] & \Aut(\Delta(\Gamma)) \arrow[r] & \Cor(\Delta(\Gamma)) \arrow[r]  & \Ct \arrow[r] & 1
    \end{tikzcd}%
    \end{center}

    is a split extension. Thus $\Cor(\Delta(\Gamma)) \cong \Aut(\Delta(\Gamma)) \times \Ct$.
\end{proof}

\subsection{The cases where $n=2$ or $m=2$}\label{betakn}
By Lemma \ref{lem:m and n}, we know that any linear space $\Gamma$ such that $\Delta(\Gamma)$ is flag-transitive must have the same number $n$ of points on every line and the same number $m$ of lines trough every point. Our theorems in the previous section took care of the cases were both $n$ and $m$ are bigger or equal to $3$. We now handle the remaining cases separately.

First, notice that if $m = 2$, there are exactly two lines through every point. Hence, $\Gamma$ is a triangle, which is a complete graphs with $3$ vertices respectively. 

If $n = 2$, $\Gamma$ is a complete graph.
Hence we can assume from now that $\Gamma$ is the geometry of a complete graph $\Kv_v$ where $v$ is the number of points of $\Gamma$. We will show that, even though $\Gamma$ has no dualities when $v\geq 4$, $\Delta(\Gamma)$ has dualities. Indeed, consider the map $\beta$ that sends $(p,L,1)$ to $(\bar{p},L,1)$, $(p,L,2)$ to $(\bar{p},L,3)$ and $(p,L,3)$ to $(\bar{p},L,2)$, where, in every case, $\bar{p}$ is the unique point of $L$ which is not $p$.

    \begin{lemma}\label{Knduality}
        Let $\Gamma$ be the geometry of the complete graph $\Kv_v$, with $v\geq 3$.
        The map $\beta$ defined above is a duality of $\Delta(\Gamma)$.
    \end{lemma}

    \begin{proof}
        Let $C = \{ (p_1,L_1,1), (p_2,L_2,2),(p_3,L_3,3) \}$ be a chamber of $\Delta(\Gamma)$. Note that $\bar{p_1} = p_3,\bar{p_2} = p_1,\bar{p_3}=p_2$. Then $$\beta(C) = \{ (\bar{p_1}, L_1,1),(\bar{p_2},L_2,3),(\bar{p_3},L_3,2)\} = \{ (p_3,L_1,1),(p_2,L_3,2),(p_1,L_2,3)\}$$ is a chamber of $\Delta(\Gamma)$. This suffices to show that $\beta$ is a bijection on the elements of $\Delta(\Gamma)$ that preserves incidence.
    \end{proof}

This shows that the thickness hypothesis in Proposition \ref{duality} is essential.

Observe now that the incidence graph of $\Delta(\K3)$ consists of six disconnected cliques and $\Aut(\Delta(\K3)) \cong \Ss$.
For $v \geq 4$, the statement of Lemma ~\ref{welldef} holds as well for $\Gamma = \Kv_v$. We prove it in the next lemma.

\begin{lemma}
    Let $\Gamma = \Kv_v = (\mathcal{P} \sqcup \mathcal{L},\star,t)$,  $v \geq 4$ and $\phi \in \Aut(\Delta(\Kv_v))$. Then
    \begin{enumerate}
\item   For every $p \in \mathcal P$ and for every $L,L' \in \text{Res}_\Gamma(p)$, we have that $\pi_1( \phi( (p,L,i) ) ) = \pi_1(\phi( (p,L',i) ))$. 
\item   For every $L \in \mathcal L$ and for every $p,p' \in \text{Res}_\Gamma(L)$, we have that $\pi_2( \phi( (p,L,i) ) ) = \pi_2(\phi( (p',L,i) ))$.
    \end{enumerate}
\end{lemma}

\begin{proof}
    If $v = 4$, a computation with {\sc Magma}~\cite{magma} shows the statements holds. Let us now assume that $v \geq 5$.
    
    For the first claim, the proof of Lemma \ref{welldef} works unchanged since $m = v-1$ and so $m-2 \neq 0$.
    
    For the second claim, let $(p,L,1)$ and $(\bar{p},L,1)$ with $\bar{p} \neq p$. We first compute the distance between $F_1 = (p,L,1)$ and $F_2 = (\bar{p},L,1)$ using only paths that cyclically alternate between type 1, then 2 and then 3 and repeat. We call it an alternating path. Any such path starting at $(p,L,1)$ and ending at $(\bar{p},L,1)$ has length a multiple of $3$ since we start at type $1$ and end at type $1$. Firstly, such path cannot have length $3$. Indeed, let $$(p,L,1),(p_2,L_2,2),(p_3,L_3,3),(\bar{p},L,1)$$ be an incidence chain in $\Delta(\Gamma)$. Since $(p_3,L_3,3) \star_{\Delta(\Gamma)}(\bar{p},L,1)$ , $p_3 = p$ so $L_2 = \big[ p,p_2 \big] = \big[ p_3,p_2 \big] = L_3$ thus contradicting  $(p_2,L_2,2) \star_{\Delta(\Gamma)} (p_3,L_3,3)$\footnote{Observe that such a path of length $6$ exists. Indeed, since $v\geq 4$ let $p_2,p_3 \in X_\Gamma / \{p,\bar{p}\}$ and construct $L_2 = \big[ p,p_2\big] , L_3 = \big[ p_2,p_3\big], L_1' = \big[\bar{p},p_3\big],L_2' = \big[ \bar{p},p_2\big]$ then
    $$(p,L,1),(p_2,L_2,2),(p_3,L_3,3),(\bar{p},L_1',1),(p_2,L_2',2),(p,L_2,3),(\bar{p},L,1)$$
    is an incidence chain in $\Delta(\Gamma)$ (see Figure~\ref{Kn6}).}.
    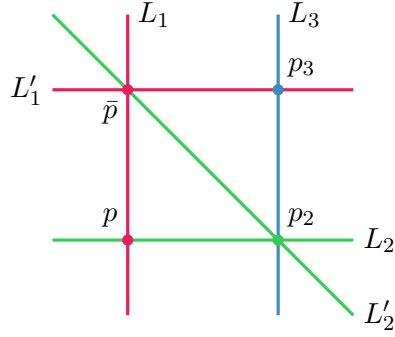
\begin{figure}
    
\begin{center}
\begin{tikzpicture}
    \draw (-3,0) node[left]{$L_1'$};
    
    \draw [red, very thick] (-3,0) -- (1,0) ;
    \draw [red, very thick] (-2,-3) -- (-2,1) ;
    \draw [blue, very thick] (0,-3) -- (0,1) ;
    \draw [green, very thick] (-3,-2) -- (1,-2) ;
    \draw [green, very thick] (-3,1) -- (1,-3) ;
    
    \draw (-2,1) node[right]{$L_1$};
    \draw (1,-2) node[right]{$L_2$};
    \draw (0,1) node[right]{$L_3$};
    \draw (1,-3) node[right]{$L_2'$};
    
    \draw (-2,-1.7) node[left]{$p$};
    \draw (-2,-0.3) node[left]{$\bar{p}$};
    \draw (0,-1.7) node[right]{$p_2$};`
    \draw (0,0.3) node[right]{$p_3$};`
    \node at (0,-2)[circle,fill,inner sep=1.5pt,color=green]{};
    \node at (-2,-2)[circle,fill,inner sep=1.5pt,color=red]{};
    \node at (0,0)[circle,fill,inner sep=1.5pt,color=blue]{};
    \node at (-2,0)[circle,fill,inner sep=1.5pt,color=red]{};
    
\end{tikzpicture}
\caption{An alternating path from $F_1$ to $F_2$ of length $6$.}\label{Kn6}
\end{center}
\end{figure}
    Assume $\pi_2(\phi(F_1)) \neq \pi_2(\phi(F_2))$.
    Let $p_3 \in (\pi_2(\phi(F_1)) \cup \pi_2(\phi(F_2))) \backslash \{\pi_1(\phi(F_1)),\pi_1(\phi(F_2))\}$. Without loss of generality assume $p_3 \in \pi_2(\phi(F_2))$. Let $p_2 \in \Kv_v \backslash (\pi_2(\phi(F_1)) \cup \pi_2(\phi(F_2)))$. Such a $p_2$ exists since $v \geq 5$. Let $L_2 = \big[ \pi(\phi(F_1)), p_2 \big]$ and $L_3 = \big[ p_2,p_3 \big]$. Then $$\phi(F_1),(p_2,L_2,2),(p_3,L_3,3),\phi(F_2)$$ is an incidence chain respecting the alternating path condition of length three, contradicting that it should be at least of length six (see Figure~\ref{Kn3}).
\begin{figure}
    
\begin{center}
\begin{tikzpicture}
    \draw [red, very thick] (2,-3) -- (2,1) ;
    \draw [red, very thick] (-2,-3) -- (-2,1) ;
    \draw [green, very thick] (-3,-2) -- (1,-2) ;
    \draw [blue, very thick] (-1,-3) -- (3,1) ;
    
    \draw (-2,1) node[right]{$\pi_2(\phi(F_1))$};
    \draw (1,-2) node[right]{$L_2$};
    \draw (2,1) node[left]{$\pi_2(\phi(F_2))$};
    \draw (3,1) node[right]{$L_3$};
    
    \draw (-2,-1.7) node[left]{$\pi_1(\phi(F_1))$};
    \draw (2,-2) node[right]{$\pi_1(\phi(F_2))$};
    \draw (0,-2) node[above]{$p_2$};`
    \draw (2,-0.2) node[right]{$p_3$};`
    \node at (0,-2)[circle,fill,inner sep=1.5pt,color=green]{};
    \node at (-2,-2)[circle,fill,inner sep=1.5pt,color=red]{};
    \node at (2,0)[circle,fill,inner sep=1.5pt,color=blue]{};
    \node at (2,-2)[circle,fill,inner sep=1.5pt,color=red]{};
    
\end{tikzpicture}
\caption{An alternating path from $\phi(F_1)$ to $\phi(F_2)$ of length $3$.}\label{Kn3}
\end{center}
\end{figure}
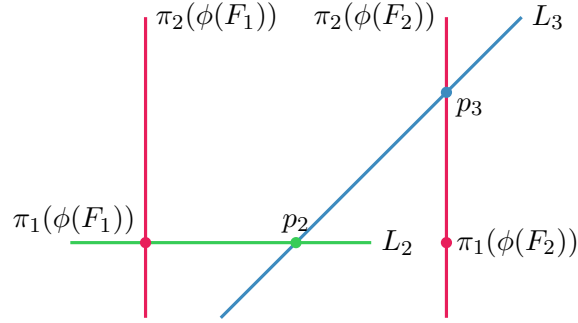
\end{proof}

This previous lemma allows us to repeat the same arguments as in Lemma \ref{bijection}, Proposition \ref{samefunctions}, Proposition \ref{masterproposition} and Theorem \ref{mastertheorem} with $\Gamma = \Kv_v$ to conclude that $\Aut(\Delta(\Kv_v)) \cong \Aut(\Kv_v) \cong \Sv_v$.
Moreover, let $\beta \in \Delta(\Kv_v)$ as in Section \ref{betakn} and let $\tau$ be the canonical triality. Then, we have that  $\St \cong \langle \tau , \beta \rangle$ and it is easily checked that $\langle \tau , \beta \rangle$ is a normal subgroup of $ \Cor(\Delta(\Kv_v))$ in this case. Hence, we obtain that $\Cor(\Delta(\Kv_v)) \cong \Sv_v \times \St$.

\section{Flag-transitive linear spaces and triangles}\label{sec:linear}
In this section we want to prove Theorem~\ref{transitive}, that classifies linear spaces having an automorphism group acting transitively on the set of triples of non-collinear points. The motivation here is given by Theorem~\ref{mastertheorem}.
 
\begin{lemma}\label{lemtriangle}
    If $\Gamma$ is a linear space and $G\leq \Aut(\Gamma)$ is transitive on the set of ordered triples of non-collinear points, then $\Gamma$ is a flag-transitive linear space and $(\Gamma,G)$ is $(2T)_1$.
\end{lemma}
\begin{proof}
    Take a flag $\{p,L\}$ of $\Gamma$ where $p$ is a point of $\Gamma$ and $L$ is a line of $\Gamma$.
    As $(\Gamma,G)$ is transitive on the set of ordered triples of non-collinear points, the stabilizer of $p$ and $L$ in $G$, namely $G_{p,L}$, must act transitively of the lines incident to $p$ and distinct from $L$. Hence the action of the stabilizer $G_p$ must be two-transitive on the lines containing $p$. Similarly, $G_{p,L}$ must be transitive on the points of $L$ distinct from $p$. Hence $G_L$ must act two-transitively on the points incident to $L$.  
\end{proof}
This lemma permits us to rely on the classification of $(2T)_1$ flag-transitive linear spaces that we recall below. 
\begin{theorem}\cite{BDL2003}\label{2t1}
Assume $\Gamma$ is a flag-transitive linear space of $v$ points with a group $G$ acting on it. Then, if $(\Gamma,G)$ is $(2T)_1$, one of the following occurs:
\begin{enumerate}
\item $\Gamma={\AG}(2,4)$, with $G=\AGammaL(1,16)$.
\item $\Gamma=\PG(n,q)$, $v=\frac{q^{n+1}-1}{q-1}$, $\PSL(n+1,q)\trianglelefteq G \le \PGammaL(n+1,q)$ with $n\ge 2$.
\item $\Gamma=\PG(3,2)$, $v=15$ with $G\cong \Alt_7$.
\item $\Gamma=\AG(n,q)$, $v=p^d=q^n$, $G=p^d:G_0$ with $\SL(n,q)\trianglelefteq G_0$, $q\ge 3$ and $n\ge 2$.
\item  $\Gamma$ is a hermitian unital $\UH(q)$, $v=q^3+1$, $\PSU(3,q)\trianglelefteq G$.
\item $\Gamma$ is a circle and $G$ is a $3$-transitive permutation group:
     \begin{enumerate}
	\item $G=\Alt_v,G=S_{v}$, $v\ge 3$.
	\item $G\trianglerighteq {\PSL}(2,q)$, $v=q+1$ and $G$ normalizes a sharply $3$-transitive permutation group.
	\item $G = \Mv_v$, $v=11,12,22,23,24$ or $G=\Aut(\Mvd)$, $v=22$.
	\item $G= \Mo$, $v=12$.
	\item $G= \Alt_7$, $v=15$.
        \item $G=\ea(2^n):G_{0}$, $v=2^n$, $G_0\trianglerighteq \SL(n,2)$ with $n\ge 2$.
	\item $G=\ea(2^4):\Alt_7$, $v=16$.
      \end{enumerate}
\end{enumerate}
The converse is also true. For any pair $(\Gamma,G)$ satisfying the conditions of one of the cases given above, the action of $G$ is $(2T)_1$.
\end{theorem}
By Lemma~\ref{lemtriangle}, in order to prove Theorem~\ref{transitive}, we only need to look at all the cases appearing in Theorem~\ref{2t1}
and check which cases give linear spaces with a group acting transitively on the set of ordered triples of non-collinear points.

\begin{proof}[Proof of Theorem~\ref{transitive}]
Let $\Gamma$ be a linear space.
    Let $T(\Gamma)$ be the set of ordered triples of non-collinear points of $\Gamma$.
    We check the cases of Theorem~\ref{2t1} one by one to see which ones remain.

    (1) In this case, $\Gamma$ has 16 points and $\AGammaL(1,16)$ is of order 960. The set $T(\Gamma)$ has size $16\cdot 15\cdot 12 = 2880$. Hence this case cannot occur.

    (2) In this case, the group $\PSL(n+1,q)$ obviously acts transitively on $T(\Gamma)$ and so do all groups $G$ with $\PSL(n+1,q)\trianglelefteq G \le \PGammaL(n+1,q)$. This becomes case (1) of Theorem~\ref{transitive}.

    (3) The cardinality of $T(\Gamma)$ is $15\cdot 14\cdot 12 = 2520$ and the group $\Alt_7$ acts transitively on $T(\Gamma)$. This becomes case (2) of Theorem~\ref{transitive}.

    (4) In that case, we need $G_0$ to be large enough so that the group $p^n:G_0$ acts transitively on triples of non-collinear points; This implies that $G_0\geq \GL(n,q)$ when $n=2$. This is case (3) of Theorem~\ref{transitive}.

    (5) The cardinality of $T(\Gamma)$ is $(q^3+1)\cdot q^3\cdot (q^3-q)$ while $|\PGammaU(3,q)| = (q^3+1)\cdot q^3\cdot (q^2-1)\cdot 2e$ where $q = p^e$. Hence the only possibilities for $G$ to be transitive on $T(\Gamma)$ is for $q = 2$ or $4$ and $G = \PGammaU(2,q)$. This is case (4) of Theorem~\ref{transitive}.

    (6) In all these cases, the group is 3-transitive on the points and lines have two points, hence no three points are collinear and $G$ acts transitively on $T(\Gamma)$. This is case (5) of Theorem~\ref{transitive}.
\end{proof}

\section{Connectedness and residual connectedness}\label{sec:rc}
In this section, we want to prove Theorem~\ref{classification}. Hence, for each linear space $\Gamma$ appearing in Theorem~\ref{transitive}, we now check which ones give a residually connected geometry using Construction~\ref{construct}.
Observe that the input for Construction~\ref{construct} is the incidence geometry. The group acting is not important. Therefore, Case (2) of Theorem~\ref{transitive} will be dealt with Case (1).

\begin{proposition}\label{connected}
    If $\Gamma = \AG(n,q)$ with $n \geq 2, q \geq 3$ or $\Gamma = \PG(n,q)$ with $n \geq 2, q \geq 2$ then $\Delta(\Gamma)$ is a connected geometry.
\end{proposition}
\begin{proof}
The fact that $\Delta(\Gamma)$ is a geometry is obvious. 
Hence we can check connectivity by checking that any two elements of type $1$ of $\Delta(\Gamma)$ are connected.

Let $(p_1,L_1,1) \in X_{\Delta(\Gamma)}$. We claim that for any $p_1' \star L_1$ and for any 
$L_1'\star p_1$ with $p_1\neq p_1'$ and $L_1 \neq L_1'$, there is an incidence chain from $(p_1,L_1,1)$ to $(p_1',L_1,1)$ and a chain from $(p_1,L_1,1)$ to $(p_1,L_1',1)$. 

For the first claim, choose a point $p_2$ not incident to $L_1$ and a point $p_3 \notin \{p_1',p_1\}$ 
        with $p_3 \star L_1$. Such a point exists as $\Gamma$ has at least $3$ points per lines. 
Let $L_2$ be the line incident to $p_1$ and $p_2$ and let $L_3$ be the line incident to $p_2$ and $p_3$ (see Figure~\ref{ag}).
Then $(p_1,L_1,1), (p_2,L_2,2), (p_3,L_3,3), (p_1',L_1,1)$ is an incidence chain in $\Delta(\Gamma)$.

\begin{figure}
\begin{center}
\begin{tikzpicture}
    \draw [blue, very thick] (-1,3) -- (3,-1) ;
    \draw [green, very thick] (-3,0) -- (3,0) ;
    \draw [red, very thick] (-3,-1) -- (1,3);

    \draw (-2.2,0) node[above]{$p_1$};
    \draw (2.2,0) node[above]{$p_2$};
    \draw (0,2.2) node[above]{$p_3$};
    \draw (-3,-1) node[left]{$p_1'$};
    \draw (1,3) node[right]{$L_1$};
    \draw (3,0) node[right]{$L_2$};
    \draw (-1,3) node[left]{$L_3$};
    
    \node at (-2,0)[circle,fill,inner sep=1.5pt,color = red]{};
    \node at (-3,-1)[circle,fill,inner sep=1.5pt,color = red]{};
    \node at (0,2)[circle,fill,inner sep=1.5pt,color = blue]{};
    \node at (2,0)[circle,fill,inner sep=1.5pt,color = green]{};
    
\end{tikzpicture}
\caption{A path from $(p_1,L_1,1)$ to $(p_1',L_1,1)$.}\label{ag}
\end{center}
\end{figure}

For the second claim, take $p_2$ a point not incident to $L_1$ nor $L_1'$. Let $p_3 \star L_1'$ with $p_3$ not incident to $L_2$. Finally let $L_3$ be the line of $\Gamma$ incident to $p_2$ and $p_3$ (see Figure~\ref{ag2}).
Then $(p_1,L_1,1), (p_2,L_2,2) , (p_3,L_3,3), (p_1,L_1',1)$ is an incidence chain in $\Delta(\Gamma)$.
\begin{figure}
\begin{center}
\begin{tikzpicture}

    \draw [red, very thick] (-3,-1) -- (0,2) ;
    \draw [blue, very thick] (3,1) -- (-1,-3) ;
    \draw [green, very thick] (-3,0) -- (3,0) ;
    \draw [red, very thick] (-3,1) -- (1,-3);
    
    \draw (-2,0.1) node[above]{$p_1$};
    \draw (2,0.1) node[above]{$p_2$};
    \draw (0,-1.9) node[above]{$p_3$};
    \draw (0,2) node[right]{$L_1$};
    \draw (3,0) node[right]{$L_2$};
    \draw (-1,-3) node[left]{$L_3$};
    \draw (-3,1) node[left]{$L_1'$};

    \node at (-2,0)[circle,fill,inner sep=1.5pt,color = red]{};
    \node at (0,-2)[circle,fill,inner sep=1.5pt, color = blue]{};
    \node at (2,0)[circle,fill,inner sep=1.5pt, color = green]{};

\end{tikzpicture}
\caption{A path from $(p_1,L_1,1)$ to $(p_1,L_1',1)$.}\label{ag2}
\end{center}
\end{figure}

Since $\Delta(\Gamma)$ is a geometry, every element of a given type is incident to two elements of the other types. Since $\Gamma$ is itself a connected (points and lines) incidence geometry, using the previous observations we can construct an incidence chain from any element to any other element of $\Delta(\Gamma)$ by composing the previous operations. Hence, $\Delta(\Gamma)$ is connected.
\end{proof}

\begin{proposition}
    Let $n\geq 3$. If $\Gamma = \AG(n,q)$ or $\Gamma = \PG(n,q)$, then $\Delta(\Gamma)$ is not residually connected
\end{proposition}
\begin{proof}
Let $F = \{(p_1,L_1,1),(p_2,L_2,2)\}$ be a rank two flag of $\Delta(\Gamma)$. An element $(p_3,L_3,3)$ incident to $F$ will have $L_3$ contained in the plane spanned by $L_1$ and $L_2$. Since there are multiple planes containing $L_1$, the residue of $(p_1,L_1,1)$ will split in distinct connected components, one for each plane containing $L_1$.
\end{proof}

\begin{proposition}\label{appprc}
If $\Gamma = \AG(2,q)$ with $q\geq 3$ or $\Gamma = \PG(2,q)$ with $q\geq 2$, then 
    $\Delta(\Gamma)$ is residually connected.
\end{proposition}

\begin{proof}
By Propositon~\ref{triality}, $\Delta(\Gamma)$ has a triality, hence it suffices to show that residues of a flag of type 1 are connected.
Moreover, since $\Delta(\Gamma)$ is a geometry, it suffices to show that, given any two elements of type 2 in the residue of a flag of type 1, there is a chain in that residue that connects them.

    Let $\Delta(\Gamma)_{(p_1,L_1,1)}$ be the residue of $(p_1,L_1,1)$. Let $(p_2,L_2,2) \in \Delta(\Gamma)_{(p_1,L_1,1)}$ and $(p_2',L_2',2) \in \Delta(\Gamma)_{(p_1,L_1,1)}$. Note that if $p_2 = p_2'$ then $L_2 = L_2'$ so nothing needs to be checked.
    
Assume first that $L_2 = L_2'$ and $p_2 \neq p_2'$.
Then the elements $(p_2,L_2,2)$ and $(p_2',L_2,2)$ clearly cannot be at distance $2$ from each other in the incidence graph of $\Delta(\Gamma)_{(p_1,L_1,1)}$.
Suppose there exist $q_3,q_3' \star L_1$ so that the lines $L_3 = \big[ p_2,q_3 \big]$ and $L_3' = \big[ p_2',q_3' \big]$ intersect. Name that intersection $p_2''$ (see Figure~\ref{l2l2}).  The chain $$(p_2,L_2,2) , (q_3,L_3,3) , (p_2'' , L_2'' , 2), (q_3',L_3',3),(p_2',L_2,2)$$
is an incidence chain in $\Gamma_{(p_1,L_1,1)}$. We conclude that the distance between $(p_2,L_2,2)$ and $(p_2',L_2,2)$ in the incidence graph of $\Gamma_{(p_1,L_1,1)}$ is 4.

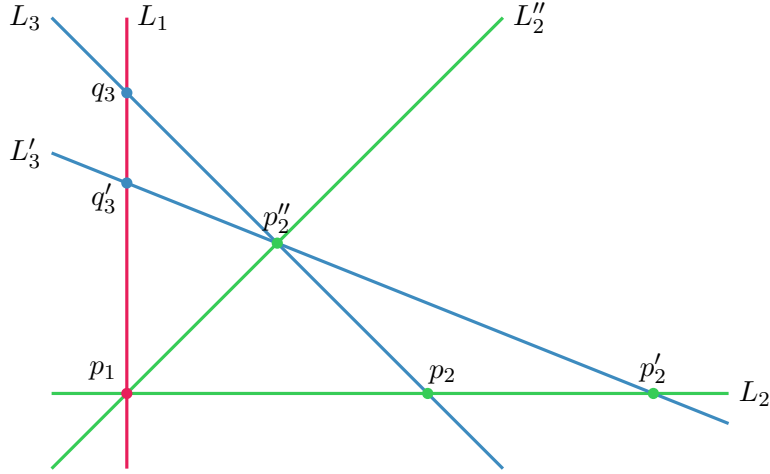
\begin{figure}
\begin{center}
    
\begin{tikzpicture}
    
    \draw [red, very thick] (-2,-3) -- (-2,3) ;
    \draw [green, very thick] (-3,-2) -- (6,-2) ;
    \draw [blue, very thick] (-3,3) -- (3,-3) ;
    \draw [green, very thick] (-3,-3) -- (3,3);
    \draw [blue, very thick] (-3,1.2) -- (6,-2.4);
    
    \draw (-2,3) node[right]{$L_1$};
    \draw (6,-2) node[right]{$L_2$};
    \draw (-3,3) node[left]{$L_3$};
    \draw (3,3) node[right]{$L_2''$};
    \draw (-3,1.2) node[left] {$L_3'$}; 
    
    \draw (-2,-1.7) node[left]{$p_1$};
    \draw (2.2,-2) node[above]{$p_2$};
    \draw (-2,2) node[left]{$q_3$};
    \draw (0,0) node[above]{$p_2''$};`
    \draw (-2,0.6) node[left]{$q_3'$};
    \draw (5,-2) node[above]{$p_2'$};`
    
    \node at (0,0)[circle,fill,inner sep=1.5pt,color=green]{};
    \node at (-2,-2)[circle,fill,inner sep=1.5pt,color=red]{};
    \node at (2,-2)[circle,fill,inner sep=1.5pt,color=green]{};
    \node at (-2,2)[circle,fill,inner sep=1.5pt,color=blue]{};
    \node at (-2,0.8)[circle,fill,inner sep=1.5pt,color=blue]{};
    \node at (5,-2)[circle,fill,inner sep=1.5pt,color=green]{};

\end{tikzpicture}
\caption{The case where $L_2 = L_2'$.}\label{l2l2}
\end{center}
\end{figure}
Suppose the above construction fails, that is we cannot find $p_3$ and $p_3'$ as requested since $L_3 = \big[ p_2,p_3 \big]$ is parallel to $L_3' = \big[p_2',p_3'\big]$. In that case, we are dealing with $\AG(2,3)$ (as if there were more than 3 points per line, we could take another point $p_3'$ to make sure $L_3$ and $L_3'$ were not parallel).
This case is small enough to be fully analysed by hand and check that the residue has then the incidence graph of a 6-gon.

Assume finally that $L_2\neq L_2'$ and $p_2\neq p_2'$.
Take $L_3 = \big[ p_2,p_2' \big]$.
If $L_3$ and $L_1$ have a point in common, let $q_3 = L_3 \cap L_1$ (see Figure~\ref{notp}). Then $(p_2,L_2,2), (q_3,L_3,3), (p_2',L_2',2)$ forms an incidence chain so $(p_2,L_2,2)$ and $(p_2',L_2',2)$ are at distance $2$.
\begin{figure}
\begin{center}
\begin{tikzpicture}
    \draw (1,1) node[right]{$L_2'$};
    
    \draw [green, very thick] (-3,-3) -- (1,1) ;
    \draw [red, very thick] (-2,-3) -- (-2,1) ;
    \draw [blue, very thick] (-3,1) -- (1,-3) ;
    \draw [green, very thick] (-3,-2) -- (1,-2) ;
    \draw (-2,1) node[right]{$L_1$};
    \draw (1,-2) node[right]{$L_2$};
    \draw (-3,1) node[left]{$L_3$};
    
    \draw (-2,-1.7) node[left]{$p_1$};
    \draw (-2,0) node[left]{$q_3$};
    \draw (0,-2) node[above]{$p_2$};`
    \draw (-1,-1) node[above]{$p_2'$};`
    \node at (0,-2)[circle,fill,inner sep=1.5pt,color=green]{};
    \node at (-2,-2)[circle,fill,inner sep=1.5pt,color=red]{};
    \node at (-1,-1)[circle,fill,inner sep=1.5pt,color=green]{};
    \node at (-2,0)[circle,fill,inner sep=1.5pt,color=blue]{};
    
\end{tikzpicture}
\caption{The case where $\big[ p_2, p_2' \big] \cap L_1 \neq \emptyset$ and $p_2' \notin L_2$. }\label{notp}
\end{center}
\end{figure}
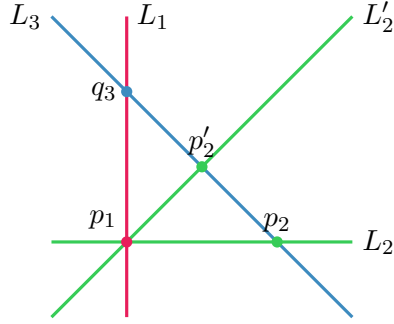

In $\AG(2,q)$, it could happen that $\big[ p_2,p_2' \big]$ is parallel to $L_1$ and the previous construction fails. 
In that case, pick a third point on $\big[ p_2,p_2' \big]$, say $p_2''$ and take the line $L_2'' = \big[ p_1,p_2'' \big]$.
On that line, take a third point $q_2$. Now, the line $L_3 = \big[p_2,q_2\big]$ necessarily intersects $L_1$ in, say $q_3$. Moreover, the line $L_3' = \big[p_2',q_2\big]$ also intersects $L_1$ in, say $q_3'$.
We have that $(p_2,L_2,2), (q_3,L_3,3), (q_2,L_2",2), (q_3',L_3',3), (p_2',L_2',2)$ is a chain in $\Delta(\Gamma)_{(p_1,L_1,1)}$, as illustrated in Figure~\ref{l2l2p}.
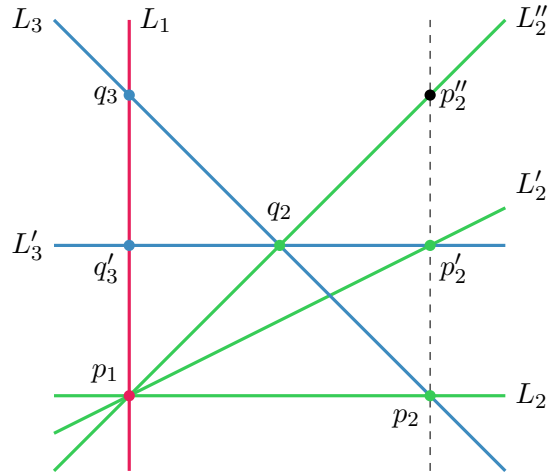
\begin{figure}
\begin{center}
    
\begin{tikzpicture}

    \draw [green, very thick] (-3,-2.5) -- (3,0.5) ;
    \draw [red, very thick] (-2,-3) -- (-2,3) ;
    \draw [green, very thick] (-3,-2) -- (3,-2) ;
    \draw [blue, very thick] (-3,3) -- (3,-3) ;
    \draw [green, very thick] (-3,-3) -- (3,3);
    \draw [dashed] (2,-3) -- (2,3);
    \draw (-2,3) node[right]{$L_1$};
    \draw (3,-2) node[right]{$L_2$};
    \draw [blue, very thick] (-3,0) -- (3,0);
    \draw (3,0.8) node[right] {$L_2'$}; 
    \draw (3,3) node[right]{$L_2''$};
    \draw (-3,0) node[left]{$L_3'$};
    \draw (-3,3) node[left]{$L_3$};
    
    \draw (-2,-1.7) node[left]{$p_1$};
    \draw (2,-2.3) node[left]{$p_2$};
    \draw (2,-0.3) node[right]{$p_2'$};
    \draw (-2,2) node[left]{$q_3$};
    \draw (2,2) node[right]{$p_2''$};`
    \draw (0,0.2) node[above]{$q_2$};
    \draw (-2,-0.3) node[left]{$q_3'$};
    \node at (0,0)[circle,fill,inner sep=1.5pt,,color=green]{};
    \node at (-2,-2)[circle,fill,inner sep=1.5pt,color=red]{};
    \node at (2,-2)[circle,fill,inner sep=1.5pt,color=green]{};
    \node at (-2,2)[circle,fill,inner sep=1.5pt,color=blue]{};
    \node at (2,0)[circle,fill,inner sep=1.5pt,color=green]{};
    \node at (-2,0)[circle,fill,inner sep=1.5pt,color=blue]{};
    \node at (2,2)[circle,fill,inner sep=1.5pt,color=black]{};

\end{tikzpicture}
\caption{The case where $\big[ p_2,p_2' \big]$ is parallel to $L_1$.}\label{l2l2p}
\end{center}
\end{figure}
\end{proof}

\begin{corollary}\label{residuegdd}
If $\Gamma = \AG(2,q)$ with $q \geq 4$ or $\Gamma = \PG(2,q)$ with $q \geq 3$, then the gonality of any rank two residue of $\Delta(\Gamma)$ is equal to $3$ and the point and line diameters are equal to $4$.
\end{corollary}

\begin{proof}
By Propositon~\ref{triality}, $\Delta(\Gamma)$ has a triality, hence we know that all rank two residues will be isomorphic. Suppose we look at a residue of a flag of type 1.
By the proof of Proposition~\ref{appprc}, we already know the distance between two elements of type $2$ in $\Delta(\Gamma)_{(p_1,L_1,1)}$ is less than or equal to $4$. 
Let $(p_2,L_2,2), (p_3,L_3,3) \in \Delta(\Gamma)_{(p_1,L_1,1)}$. We claim that $d_{\Delta(\Gamma)_{(p_1,L_1,1)}}((p_2,L_2,2),(p_3,L_3,3)) \leq 3$. Indeed, take $q_3 \in L_1$ with $q_3\notin \{ p_1,p_3 \}$ such that $\big[ p_2, q_3 \big]$ is not parallel to $L_3$.
(note that this is not possible in $\AG(2,3)$). Construct in order $L_3' = \big[ p_2,q_3 \big]$, $p_2' = L_3' \cap L_3$, $L_2' = \big[ p_2',p_1 \big]$. Then 
$$(p_2,L_2,2),(q_3,L_3',3),(p_2',L_2',2),(p_3,L_3,3)$$
is an incidence chain in $\Delta(\Gamma)_{(p_1,L_1,1)}$ (see Figure~\ref{dist3}). 
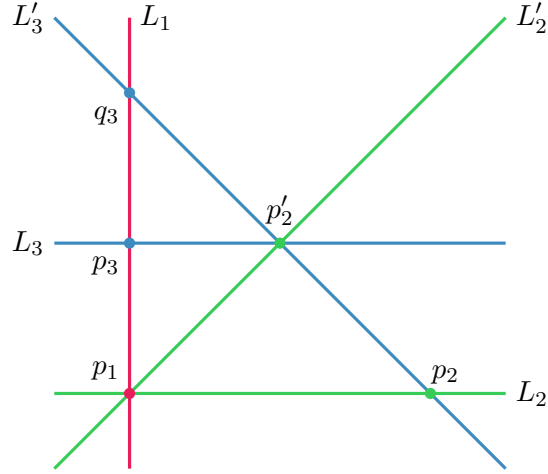
\begin{figure}
\begin{center}
    
\begin{tikzpicture}
    \draw [red, very thick] (-2,-3) -- (-2,3) ;
    \draw [green, very thick] (-3,-2) -- (3,-2) ;
    \draw [blue, very thick] (-3,3) -- (3,-3) ;
    \draw [green, very thick] (-3,-3) -- (3,3);
    \draw [blue, very thick] (-3,0) -- (3,0);

    \draw (-2,-1.7) node[left]{$p_1$};
    \draw (2.2,-2) node[above]{$p_2$};
    \draw (-2,1.7) node[left]{$q_3$};
    \draw (0,0.1) node[above]{$p_2'$};
    \draw (-2,-0.3) node[left]{$p_3$};`
    
    \node at (0,0)[circle,fill,inner sep=1.5pt,color = green]{};
    \node at (-2,-2)[circle,fill,inner sep=1.5pt,color = red]{};
    \node at (2,-2)[circle,fill,inner sep=1.5pt,color = green]{};
    \node at (-2,2)[circle,fill,inner sep=1.5pt,color = blue]{};
    \node at (-2,0)[circle,fill,inner sep=1.5pt,color = blue]{};
    
    \draw (-2,3) node[right]{$L_1$};
    \draw (3,-2) node[right]{$L_2$};
    \draw (-3,0) node[left]{$L_3$};
    \draw (3,3) node[right]{$L_2'$};
    \draw (-3,3) node[left] {$L_3'$};    
    
\end{tikzpicture}
\caption{A path of length three.}\label{dist3}
\end{center}

\end{figure}
Hence the point and line diameter are equal to $4$.

To compute the gonality, notice we have at least $4$ points per lines. Choose three points $q_3 , q_3', q_3'' \star L_1$ and $q_3,q_3',q_3'' \neq p_1$. Construct the lines $L_{3} = \big[ q_3 , p_2 \big]$ and $L_{3}' = \big[ q_3' , p_2 \big]$. There are at least $2$ points in $L_{3}' / \{ q_3' , p_2 \}$ so at least one of the two lines passing by $q_3''$ and the first point and $q_3''$ and the second point will be non parallel to $L_{3}$ (you do not need to worry about this case in the projective plane). Name that line $L_{3}''$. Let $p_2' = L_3' \cap L_3''$ and $p_2'' = L_3 \cap L_3''$. Then $$(p_2,L_2,2),(q_3,L_{3},3),(p_2'',L_2'',2)(q_3'',L_{3}'',3),(p_2',L_2',2),(q_3',L_{3}',3),(p_2,L_2,2)$$ is an incidence chain of length $6$ in $\Delta(\Gamma)_{(p_1,L_1,1)}$ (see Figure~\ref{gon}).

One can easily verify that cycles based in $(p_2,L_2,2)$ of length $4$ cannot exist in $\Delta(\Gamma)_{(p_1,L_1,1)}$. Moreover, the same construction can be used for cycles with basepoint $(p_3,L_3,3)$. We conclude that the gonality is $3$ which completes the proof.
\begin{figure}

\begin{center}
  
\begin{tikzpicture}[scale = 0.7]
    
    \draw [blue, very thick] (-3,4+6/7) -- (6,-2-6/7) ;
    \draw [blue, very thick] (-3,1+3/7) -- (6,-2-3/7) ;
    \draw [blue, very thick] (-3,-51/8) -- (3+2.0635,158/80+2.0635*1.3967) ;
    
    \draw [green, very thick](-3,-2) -- (6,-2) ;
    \draw [green, very thick](-3,-2.5) -- (1.32+4.73,-0.4+4.73*0.468111);
    \draw [green, very thick](-3,-2.7) -- (2+4,0.6+2.64);

    \draw [red, very thick](-2,4+6/7) -- (-2,-51/8);
    
    \draw (-3,4+6/7) node[left]{$L_{3}$};
    \draw (-3,1+3/7) node[left]{$L_{3}'$};
    \draw (6,-2) node[right]{$L_{2}$};
    \draw (-3,-51/8) node[left]{$L_{3}''$};
    \draw (1.32+4.73,-0.4+4.73*0.468111) node[right]{$L_2'$};
    \draw (2+4,0.6+2.64) node[right]{$L_2''$};
    \draw (-2,4+6/7) node[right]{$L_1$};
    
    \draw (-2,3.7) node[left]{$q_3$};
    \draw (-2,0.7) node[left]{$q_3'$};
    \draw (-2,-1.7) node[left]{$p_{1}$};
    \draw (-2,-4.7) node[left]{$q_3''$};
`   \draw (2,0.9) node[above]{$p_{2}''$};
    \draw (1.4,-0.6) node[below]{$p_{2}'$};
    \draw (5.3,-2) node[above]{$p_2$};

    \node at (1.32,-0.4)[circle,fill,inner sep=1.5pt,color=green]{};
    \node at (-2,4)[circle,fill,inner sep=1.5pt,color=blue]{};
    \node at (-2,1)[circle,fill,inner sep=1.5pt,color=blue]{};
    \node at (-2,-2)[circle,fill,inner sep=1.5pt,color=red]{};
    \node at (-2,-5)[circle,fill,inner sep=1.5pt,color=blue]{};
    \node at (2,0.6)[circle,fill,inner sep=1.5pt,color=green]{};
    
    \node at (5,-2)[circle,fill,inner sep=1.5pt,color=green]{};
    
\end{tikzpicture}
\caption{A cycle with basepoint $(p_2,L_2,2)$ of length six in $\Delta(\Gamma)_{(p_1,L_1,1)}$.}\label{gon}
\end{center}

\end{figure}
\end{proof}

\begin{proposition}
    If $\Gamma = \Kv_v$, $v \geq 4$ then $\Delta(\Gamma)$ is connected and not residually connected.
\end{proposition}
\begin{proof}
Since every line has exactly two points, $\Delta(\Gamma)$ is not residually connected. 
We first prove there is an incidence chain from $(p,L, i )$ to $(p,L, i \bmod{3} +1)$ in $\Delta(\Gamma)$. Secondly, we prove there is an incidence chain from $(p,L,i)$ to $(\bar{p},L,i)$ in $\Delta(\Gamma)$. 
Without loss of generality, take $(p_1,L_1,1) \in X_{\Delta(\Gamma)}$. Let $p_2,p_3 \notin \text{Res}_\Gamma(L_1)$.

For the first claim, construct $L_2 = \big[p_1,p_2\big], L_3 = \big[p_2,p_3\big],L_1' = \big[p_3,\bar{p_1} \big]$. Then, $$(p_1,L_1,1),(p_2,L_2,2),(p_3,L_3,3),(\bar{p_1},L_1',1),(p_1,L_1,2)$$ is an incidence chain in $\Delta(\Gamma)$ (see Figure~\ref{Kncon1}).

\begin{figure}
\begin{center}
\begin{tikzpicture}
    \draw (-3,0) node[left]{$L_1'$};
    
    \draw [red, very thick] (-3,0) -- (1,0) ;
    \draw [red, very thick] (-2,-3) -- (-2,1) ;
    \draw [blue, very thick] (0,-3) -- (0,1) ;
    \draw [green, very thick] (-3,-2) -- (1,-2) ;
    \draw (-2,1) node[right]{$L_1$};
    \draw (1,-2) node[right]{$L_2$};
    \draw (0,1) node[right]{$L_3$};
    
    \draw (-2,-1.7) node[left]{$p_1$};
    \draw (-2,0.3) node[left]{$\bar{p_1}$};
    \draw (0,-1.7) node[right]{$p_2$};`
    \draw (0,0.3) node[right]{$p_3$};`
    \node at (0,-2)[circle,fill,inner sep=1.5pt,color=green]{};
    \node at (-2,-2)[circle,fill,inner sep=1.5pt,color=red]{};
    \node at (0,0)[circle,fill,inner sep=1.5pt,color=blue]{};
    \node at (-2,0)[circle,fill,inner sep=1.5pt,color=red]{};
    
\end{tikzpicture}
\caption{A path from $(p_1,L_1,1)$ to $(p_1,L_1,2).$}\label{Kncon1}
\end{center}
\end{figure}

For the second claim, construct $L_2 = \big[ p_1,p_2 \big], L_3 = \big[ p_2,p_3 \big], L_1' = \big[p_3,p_1\big]$. Then, $$(p_1,L_1,1),(p_2,L_2,2),(p_3,L_3,3),(p_1,L_1',1),(\bar{p_1},L_1,2)$$ is an incidence chain in $\Delta(\Gamma)$ (see Figure~\ref{Kncon2}). By then using the first claim two times we get an incidence chain from $(p_1,L_1,1)$ to $(\bar{p_1},L_1,1)$.

\begin{figure}
\begin{center}
\begin{tikzpicture}
    \draw (-3,-3) node[left]{$L_1'$};
    
    \draw [red, very thick] (-3,-3) -- (1,1) ;
    \draw [red, very thick] (-2,-3) -- (-2,1) ;
    \draw [blue, very thick] (0,-3) -- (0,1) ;
    \draw [green, very thick] (-3,-2) -- (1,-2) ;
    \draw (-2,1) node[right]{$L_1$};
    \draw (1,-2) node[right]{$L_2$};
    \draw (0,1) node[right]{$L_3$};
    
    \draw (-2,-1.7) node[left]{$p_1$};
    \draw (-2,0) node[left]{$\bar{p_1}$};
    \draw (0,-1.7) node[right]{$p_2$};`
    \draw (0,-0.2) node[right]{$p_3$};`
    \node at (0,-2)[circle,fill,inner sep=1.5pt,color=green]{};
    \node at (-2,-2)[circle,fill,inner sep=1.5pt,color=red]{};
    \node at (0,0)[circle,fill,inner sep=1.5pt,color=blue]{};
    \node at (-2,0)[circle,fill,inner sep=1.5pt,color=green]{};
    
\end{tikzpicture}
\caption{A path from $(p_1,L_1,1)$ to $(p_1,L_1,2).$}\label{Kncon2}
\end{center}
\end{figure}
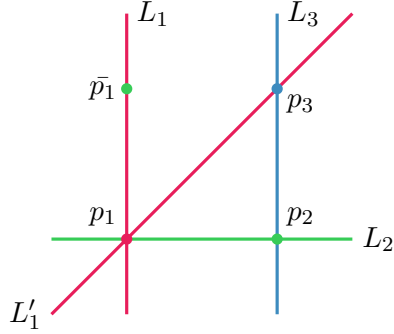

Using these two moves, it is readily seen that $\Delta(\Gamma)$ is connected.

\end{proof}

\begin{proof}[Proof of Theorem~\ref{classification}]Theorem~\ref{classification} is a summary of the results obtained in the Propositions and Corollary of this section.
\end{proof}
 \begin{corollary}\label{cor: infinite family}
     The family $\Delta(\AG(2,q))$ for $q > 3$, $q$ a power of a prime, is an infinite family of thick, residually connected and flag-transitive geometries with trialities but no dualities.
 \end{corollary}
\begin{proof}
    This is a direct consequence of Theorem~\ref{classification} and Proposition~\ref{prop:dual}.    
\end{proof}

\section{Firm, residually connected and flag-transitive geometries with a triality}\label{sec:diagrams}
We now give the Buekenhout diagrams of the geometries appearing in Theorem~\ref{classification}.

For $\Gamma = \PG(2,2)$,  the diagram of $\Delta(\Gamma)$ can be computed either by hand or using {\sc Magma}~\cite{magma} and is as follows:

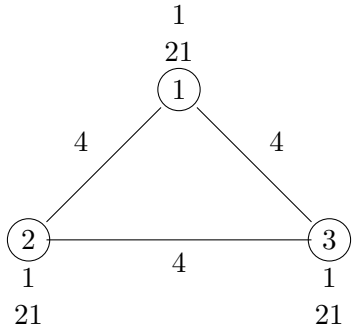
\begin{center}

\begin{tikzpicture}

\begin{scope}[every node/.style={thick}]
    \draw (0,0) circle (8pt);
    \draw (-2,-2) circle (8pt);
    \draw (2,-2) circle (8pt);
    
    \node (1) at (0,0) {$1$};
    \node (2) at (-2,-2) {$2$};
    \node (3) at (2,-2) {$3$};
    \node (11) at (0,1) {$1$};
    \node (12) at (0,0.5) {$21$};
    \node (21) at (-2,-2.5) {$1$};
    \node (22) at (-2,-3) {$21$};
    \node (31) at (2,-2.5) {$1$};
    \node (32) at (2,-3) {$21$};

    \node (4) at (-1.3,-0.7) {$4$};
    \node (5) at (1.3,-0.7) {$4$};
    \node (6) at (0,-2.3) {$4$};
\end{scope}

\begin{scope}[every node/.style={thick}]
    
    \path [-] (1) edge[color=black] (3);
    \path [-] (2) edge[color=black] (3);
    \path [-] (2) edge[color=black] (1);
\end{scope}

\end{tikzpicture}
\end{center}

For $\Gamma = \PG(2,q)$, with $q \geq 3$ the diagram of $\Delta(\Gamma)$ is computed thanks to Corollary~\ref{residuegdd}.

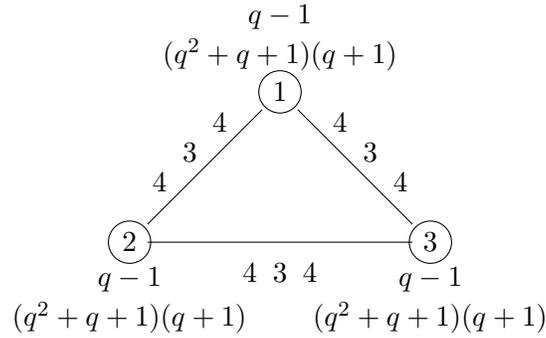
\begin{center}

\begin{tikzpicture}

\begin{scope}[every node/.style={thick}]
    \draw (0,0) circle (8pt);
    \draw (-2,-2) circle (8pt);
    \draw (2,-2) circle (8pt);
    
    \node (1) at (0,0) {$1$};
    \node (2) at (-2,-2) {$2$};
    \node (3) at (2,-2) {$3$};
    \node (11) at (0,1) {$q-1$};
    \node (12) at (0,0.5) {$(q^2+q+1)(q+1)$};
    \node (21) at (-2,-2.5) {$q-1$};
    \node (22) at (-2,-3) {$(q^2+q+1)(q+1)$};
    \node (31) at (2,-2.5) {$q-1$};
    \node (32) at (2,-3) {$(q^2+q+1)(q+1)$};

    \node (41) at (-1.2,-0.8) {$3$};
    \node (42) at (-1.6,-1.2) {$4$};
    \node (43) at (-0.8,-0.4) {$4$};
    \node (51) at (0.8,-0.4) {$4$};
    \node (52) at (1.2,-0.8) {$3$};
    \node (53) at (1.6,-1.2) {$4$};
    \node (61) at (0,-2.4) {$3$};
    \node (62) at (-0.4,-2.4) {$4$};
    \node (63) at (0.4,-2.4) {$4$};

\end{scope}

\begin{scope}[every node/.style={thick}]
    
    \path [-] (1) edge[color=black] (3);
    \path [-] (2) edge[color=black] (3);
    \path [-] (2) edge[color=black] (1);
\end{scope}

\end{tikzpicture}
\end{center}

For $\Gamma = \AG(2,3)$, the diagram of $\Delta(\Gamma)$ can again be computed either by hand or using {\sc Magma}~\cite{magma}.

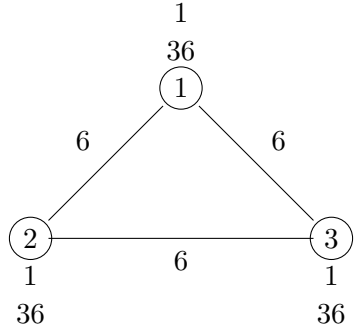
\begin{center}

\begin{tikzpicture}

\begin{scope}[every node/.style={thick}]
    
    \draw (0,0) circle (8pt);
    \draw (-2,-2) circle (8pt);
    \draw (2,-2) circle (8pt);
    
    \node (1) at (0,0) {$1$};
    \node (2) at (-2,-2) {$2$};
    \node (3) at (2,-2) {$3$};
    \node (11) at (0,1) {$1$};
    \node (12) at (0,0.5) {$36$};
    \node (21) at (-2,-2.5) {$1$};
    \node (22) at (-2,-3) {$36$};
    \node (31) at (2,-2.5) {$1$};
    \node (32) at (2,-3) {$36$};

    \node (4) at (-1.3,-0.7) {$6$};
    \node (5) at (1.3,-0.7) {$6$};
    \node (6) at (0,-2.3) {$6$};
\end{scope}

\begin{scope}[every node/.style={thick}]
    
    \path [-] (1) edge[color=black] (3);
    \path [-] (2) edge[color=black] (3);
    \path [-] (2) edge[color=black] (1);
\end{scope}

\end{tikzpicture}
\end{center}

For $\Gamma = \AG(2,q), q \geq 4$ the diagram of $\Delta(\Gamma)$ is the following thanks to Corollary~\ref{residuegdd}.

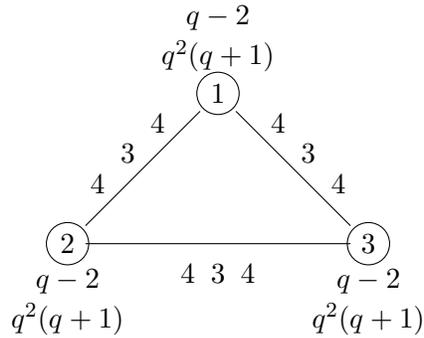
\begin{center}

\begin{tikzpicture}

\begin{scope}[every node/.style={thick}]
    \draw (0,0) circle (8pt);
    \draw (-2,-2) circle (8pt);
    \draw (2,-2) circle (8pt);
    
    \node (1) at (0,0) {$1$};
    \node (2) at (-2,-2) {$2$};
    \node (3) at (2,-2) {$3$};
    \node (11) at (0,1) {$q-2$};
    \node (12) at (0,0.5) {$q^2(q+1)$};
    \node (21) at (-2,-2.5) {$q-2$};
    \node (22) at (-2,-3) {$q^2(q+1)$};
    \node (31) at (2,-2.5) {$q-2$};
    \node (32) at (2,-3) {$q^2(q+1)$};

    \node (41) at (-1.2,-0.8) {$3$};
    \node (42) at (-1.6,-1.2) {$4$};
    \node (43) at (-0.8,-0.4) {$4$};
    \node (51) at (0.8,-0.4) {$4$};
    \node (52) at (1.2,-0.8) {$3$};
    \node (53) at (1.6,-1.2) {$4$};
    \node (61) at (0,-2.4) {$3$};
    \node (62) at (-0.4,-2.4) {$4$};
    \node (63) at (0.4,-2.4) {$4$};

\end{scope}

\begin{scope}[every node/.style={thick}]
    
    \path [-] (1) edge[color=black] (3);
    \path [-] (2) edge[color=black] (3);
    \path [-] (2) edge[color=black] (1);
\end{scope}

\end{tikzpicture}
\end{center}

For $\Gamma = \UH(2)$, as $\UH(2)\cong \AG(2,3)$, the diagram of $\Delta(\Gamma)$ is already given above.
For $\Gamma = \UH(4)$, the diagram of $\Delta(\Gamma)$ diagram can be computed thanks to {\sc Magma}. The rank 2 residues do not have dualities.

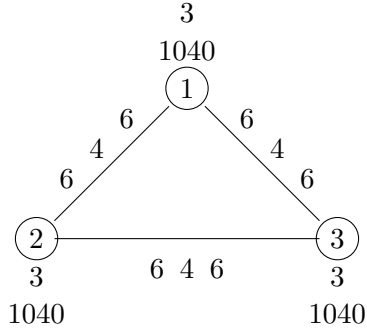
\begin{center}

\begin{tikzpicture}

\begin{scope}[every node/.style={thick}]

    \draw (0,0) circle (8pt);
    \draw (-2,-2) circle (8pt);
    \draw (2,-2) circle (8pt);

    \node (1) at (0,0) {$1$};
    \node (2) at (-2,-2) {$2$};
    \node (3) at (2,-2) {$3$};
    \node (11) at (0,1) {$3$};
    \node (12) at (0,0.5) {$1040$};
    \node (21) at (-2,-2.5) {$3$};
    \node (22) at (-2,-3) {$1040$};
    \node (31) at (2,-2.5) {$3$};
    \node (32) at (2,-3) {$1040$};

    \node (41) at (-1.2,-0.8) {$4$};
    \node (42) at (-1.6,-1.2) {$6$};
    \node (43) at (-0.8,-0.4) {$6$};
    \node (51) at (0.8,-0.4) {$6$};
    \node (52) at (1.2,-0.8) {$4$};
    \node (53) at (1.6,-1.2) {$6$};
    \node (61) at (0,-2.4) {$4$};
    \node (62) at (-0.4,-2.4) {$6$};
    \node (63) at (0.4,-2.4) {$6$};

\end{scope}

\begin{scope}[every node/.style={thick}]
    
    \path [-] (1) edge[color=black] (3);
    \path [-] (2) edge[color=black] (3);
    \path [-] (2) edge[color=black] (1);
\end{scope}

\end{tikzpicture}
\end{center}

We conclude this section with Table~\ref{table}. This table lists all the geometries $\Gamma$ we discussed and tells us whether $\Delta(\Gamma)$ is connected, residually connected or thin. It also contains the description of the automorphism group and correlation group of $\Delta(\Gamma)$.

\begin{center}
\begin{table}[ht]
\begin{tabular}{||c c c c c c||} 
 \hline
 $\Gamma$ & Connected & RC & Thin & $\Aut(\Delta(\Gamma))$ & $\Cor(\Delta(\Gamma))$ \\ [0.5ex] 
 \hline\hline
 $\AG(2,3)$ & Yes & Yes & Yes & $\AGL(2,3)$ & $\AGL(2,3) \times \Ct$ \\ 
 \hline
 $\AG(n,3)$, $n \geq 3$ & Yes & No & Yes & $\AGL(n,3)$ & $\AGL(n,3) \times  \Ct$ \\
 \hline
 $\AG(n,q)$, $n \geq 3, q \geq 4$ & Yes & No & No & $\AGL(n,q)$ & $\AGL(n,q) \times  \Ct$ \\
 \hline
 $\AG(2,q)$, $q \geq 4$ & Yes & Yes & No & $\AGL(2,q)$ & $\AGL(2,q) \times \Ct$ \\
 \hline
 $\PG(2,2)$ & Yes & Yes & Yes & $\PGammaL(3,2)$ & $\PGammaL(3,2) \rtimes \St$ \\ 
 \hline
 $\PG(2,q)$, $q \geq 3$ & Yes & Yes & No & $\PGammaL(3,q)$ & $\PGammaL(3,q) \rtimes \St$ \\
 \hline
 $\PG(n,q)$, $n \geq 3, q \geq 3$ & Yes & No & No & $\PGammaL(n+1,q)$ & $\PGammaL(n+1,q) \times \Ct$ \\
 \hline
 $\PG(3,2)$ & Yes & No & Yes & $\PGammaL(4,2)$ & $\PGammaL(4,2) \times \Ct$ \\
 \hline
 $\K3$ & No & Yes & No & $\Ss$ & $\Ss \times \St$ \\
 \hline
 $\Kn$, $n \geq 4$ & Yes & No & No & $\Sn$ & $\Sn \times \St$ \\
 \hline
 $\UH(4)$& Yes & Yes & No & $\PGammaU(3,4)$ & $\PGammaU(3,4) \times \Ct$ \\
 [0ex] 
 \hline

\end{tabular}
\caption{}
\label{table}
\end{table}

\end{center}

\section{Conclusion}\label{sec:conclusion}

In this paper, we described a geometric construction that permits to get, from a rank two geometry $\Gamma$, a rank three incidence system $\Delta(\Gamma)$ that admits trialities. 
The gonality of $\Gamma$ needs to be at most three as proven in Proposition~\ref{g3impliesgeom} in order to have that $\Delta(\Gamma)$ be a geometry.
We then focused on the case where the gonality is three and this led us naturally to investigate which linear spaces $\Gamma$ give interesting geometries $\Delta(\Gamma)$.
We thus determined which linear spaces $\Gamma$ give "interesting geometries" $\Delta(\Gamma)$, that is geometries that are firm, residually connected and flag-transitive.

We did not investigate the case where the gonality of $\Gamma$ is two. However, experiments suggest that $\Delta(\Gamma)$ will not be residually connected in that case. 

We did not find any examples of rank two geometries $\Gamma$ that have gonality three but are not linear spaces and that would give a $\Delta(\Gamma)$ that is firm, residually connected and flag-transitive.
The construction suggests however that there is no restriction on the diameter of $\Gamma$, but simply that elements at distance more than four of each other will not end up in common rank two residues.

\bibliographystyle{abbrv} 
\bibliography{refs}
\end{document}